\PassOptionsToPackage{unicode}{hyperref}
\PassOptionsToPackage{naturalnames}{hyperref}

\documentclass[a4paper,11pt]{article}

\usepackage{etex}
\usepackage{graphicx}
\usepackage{amsmath}
\usepackage{amssymb}
\usepackage{amsthm}
\usepackage{empheq}
\usepackage{multirow}
\usepackage{makecell}
\usepackage{pict2e}
\usepackage{booktabs}
\usepackage[figuresright]{rotating}
\usepackage{bbm}
\usepackage{tikz}

\usetikzlibrary{calc,trees,positioning,arrows,chains,shapes.geometric,%
    decorations.pathreplacing,decorations.pathmorphing,shapes,%
    matrix,shapes.symbols,external,fadings,patterns}

\usepackage{algorithm}      
\usepackage{pgfplots}
\usepackage{listings}		
\usepackage[noend]{algpseudocode}
\usepackage{changepage}     
\usepackage{color}              
\definecolor{myred}{gray}{0} 
\definecolor{light-gray}{gray}{0.8}

\usepackage[caption=false]{subfig}
\usepackage{epstopdf}
\usepackage{graphics} 
\usepackage{epsfig} 
\usepackage[ plainpages = false, pdfpagelabels,
	     bookmarks,
	     bookmarksopen = true,
	     bookmarksnumbered = true,
	     breaklinks = true,
	     linktocpage,
	     pagebackref,         
	     colorlinks = false,  
         hypertexnames=false,   
	     linkcolor = blue,
	     urlcolor  = blue,
	     citecolor = red,
	     anchorcolor = green,
	     hyperindex = true,
	     hyperfigures
	     ]{hyperref}
\usepackage{pdflscape}
\usepackage{todonotes}\reversemarginpar
\usepackage{overpic}
\makeatletter
\newenvironment{subfigures}
 {\begin{minipage}{\columnwidth}\def\@captype{figure}\centering}
 {\end{minipage}}
\makeatother

\newtheorem{theorem}{Theorem}[section]

\theoremstyle{remark}
\newtheorem{remark}[theorem]{Remark}

\theoremstyle{definition}

\newcommand{\vect}[1]{#1} 
\newcommand{\Vect}[1]{{\bf #1}} 

\newcommand\Sect{Section } 
\newcommand\Sects{Sections }
\newcommand\Eqn{Equation }
\newcommand\Eqns{Equations }

\newcommand\Fig{figure }

\newcommand\Table{Table }

\newcommand\vel{u}          
\newcommand\Var{w}   

\DeclareMathAlphabet\mathbfcal{OMS}{cmsy}{b}{n}

\newcommand\transitionA{\mathbfcal{A}}
\newcommand\systemA{\mathbfcal{A}}

\let\originalleft\left
\let\originalright\right
\renewcommand{\left}{\mathopen{}\mathclose\bgroup\originalleft}
\renewcommand{\right}{\aftergroup\egroup\originalright}

\binoppenalty=\maxdimen 
\relpenalty=\maxdimen   

\numberwithin{equation}{section}    
\date{\today}             

\title{Spatially Adaptive Projective Integration Schemes \\ For Stiff Hyperbolic Balance Laws With Spectral Gaps}
\author{
Julian Koellermeier\footnote{Corresponding author, email address {\tt julian.koellermeier@kuleuven.be}} \footnote{Department of Computer Science, KU Leuven},
Giovanni Samaey\footnotemark[\value{footnote}]
}

\begin{document}

\maketitle



\begin{abstract}
    Stiff hyperbolic balance laws exhibit large spectral gaps, especially if the relaxation term significantly varies in space. Using examples from rarefied gases and the general form of the underlying balance law model, we perform a detailed spectral analysis of the semi-discrete model that reveals the spectral gaps. Based on that, we show the inefficiency of standard time integration schemes expressed by a severe restriction of the CFL number. We then develop the first spatially adaptive projective integration schemes to overcome the prohibitive time step constraints of standard time integration schemes. The new schemes use different time integration methods in different parts of the computational domain, determined by the spatially varying value of the relaxation time. We use our analytical results to derive accurate stability bounds for the involved parameters and show that the severe time step constraint can be overcome. The new adaptive schemes show good accuracy in a numerical test case and can obtain a large speedup with respect to standard schemes.
\end{abstract}

\section{Introduction}
\label{sec:1}
Many applications from science and engineering are modeled by partial differential equations in balance law form, often including stiff relaxation terms with different time scales. Typical examples can be found in aerodynamics, rarefied gases, and atmospheric flows \cite{Anderson1991,Klein2010,Struchtrup2006,Torrilhon2016}. The relaxation terms in the modeled equations can lead to a time scale separation governed by a small parameter, called the relaxation time \cite{Gear2003}. This relaxation time can vary largely throughout the domain, giving rise to different modes developing on different time scales \cite{Koellermeier2021}.

For a large time scale separation, indicated by a spectral gap of the eigenvalue spectrum, the equations become stiff after the spatial discretization. For standard explicit time stepping schemes, this stiffness leads to prohibitively small time steps that are  proportional to the relaxation time. This is problematic in many applications including models from rarefied gases derived with the help of kinetic theory \cite{Fan2016,Koellermeier2017d}.
In the limit of vanishing relaxation time, an asymptotic preserving scheme is necessary for a stable computational simulation and feasible runtime \cite{Jin2010}.

It is possible to use an implicit scheme for the discretization of the semi-discrete PDE. However, fully implicit solutions are very expensive and not appropriate for hyperbolic fluid dynamics problems \cite{Lafitte2017}. Hybrid schemes like the implicit-explicit IMEX schemes \cite{Pareschi2005} still have a remaining implicit term that might require a special treatment. Splitting the equations into a stiff and a non-stiff term is effective for lower order schemes, but cannot be easily generalized to higher order \cite{Strang1968,Tcheremissine2001}.

Projective integration (PI) is a simple to implement, explicit time integration scheme that mitigates stiffness problems by performing a number $K$ of small inner time steps of size $\delta t$ followed by a large extrapolation step of size $\Delta t$ \cite{Gear2003}. The small time step size $\delta t$ is fixed by the fast relaxation time whereas the extrapolation step $\Delta t$ can be chosen according to a macroscopic time step governed by a standard $CFL$ condition. The number of inner iterations $K$ is then determined by the necessary stability properties, but often taken as a small fixed number. The scheme has been successfully applied to different models arising from kinetic equations \cite{Lafitte2010,Melis2017}. It was extended to higher-order in space and time using a Runge-Kutta scheme as outer integrator in \cite{Lafitte2016,Lafitte2017} and a telescopic scheme with different levels of integrators was developed and applied in \cite{Melis2019,Melis2016}. Recently, the models have been applied to hyperbolic moment models in \cite{Koellermeier2020h,Koellermeier2021}.

So far, all existing PI schemes use constant method parameters $K, \delta t, \Delta t$ throughout the whole spatial domain. This approach is optimized for a spatially constant relaxation time. The existing methods are therefore not able to take advantage of differences in the relaxation times in different parts of the domain. While the PI scheme is efficient in the stiff region, it might not be needed in parts of the domain with large relaxation time. 

In this paper, we introduce the first spatially adaptive projective integration schemes (API) using a domain decomposition approach with buffer cells at the respective boundaries. To the best of our knowledge, no such scheme was described in the literature so far. In the methods we propose, each part of the domain uses a different time step size and a potentially different time integration scheme, based on varying relaxation times throughout parts of the domain and different time scales. This leads to a significant speedup in non-stiff regions and reduces the stability constraints in each part of the domain. The buffer cells for reconstruction at the boundary are updated based on interpolation between two adjacent time levels. Similar approaches have been used before, e.g., coupling Lattice Boltzmann models or rarefied gas models and other PDEs \cite{Degond2005,Garcia1999,Kolobov2007,Tiwari1998,VanLeemput2006,Xiao2020}.

The main focus of this paper is the stability analysis of those newly derived API schemes to ensure a stable integration of the model throughout the whole domain using adaptively chosen parameters for the different schemes in the different parts of the domain. This requires an in-depth stability analysis of the semi-discrete model before performing the time discretization. We consider a general balance law form with standard spatial discretization schemes and derive the spectrum of the model based on some simple assumptions on the relaxation term and the properties of the transport term. This will then allow to perform a linear stability analysis for a large class of standard non-adaptive schemes and newly derived adaptive schemes. As examples, we derive the following new schemes: (1) an adaptive Forward Euler scheme (AFE); (2) an adaptive Projective Forward Euler scheme combined with a Forward Euler scheme in the semi-stiff region (APFE); and (3) an adaptive Projective Projective Forward Euler scheme combined with a projective Forward Euler scheme in the semi-stiff region (APPFE). In addition, we outline the extension to possible higher-order adaptive Projective Runge Kutta schemes (APRK) or adaptive Telescopic Projective Integration schemes (ATPI), similar to \cite{Lafitte2017,Melis2019}. After analytically deriving stable parameter bounds for all schemes in the presence of one stiff domain part and one semi-stiff domain part, we numerically show that these parameters indeed lead to a stable scheme. We derive analytical estimates for the speedup of the new adaptive schemes with respect to a standard global Forward Euler scheme or global Projective Forward Euler scheme.


The rest of this paper is structured as follows: In \Sect \ref{sec:2}, we first introduce the general type of stiff hyperbolic balance law and give two examples for models from rarefied gases. Additionally, we describe the spatial discretization based on standard finite volume schemes. The spatial discretization of standard models allows for a detailed spectral analysis that reveals a clear spectral gap, which is also validated numerically. \Sect \ref{sec:standard_schemes} considers standard FE and PFE schemes defined in the whole domain for which a restrictive CFL condition is analytically derived. The adaptive schemes are derived in \Sect \ref{sec:adaptive_schemes} along with an analysis of the stability properties and a numerical validation of the theoretical results. A numerical test case is solved before speedup estimates are derived and exemplified in \Sect \ref{sec:results}. 
The paper ends with a short conclusion.     
\section{Stiff hyperbolic balance laws with spectral gaps}
\label{sec:2}
Many problems in science and engineering can be modelled as balance laws with a left-hand side transport term and a right-hand side relaxation term, which can be interpreted as a source term. In this paper, we consider non-conservative systems of the following form
\begin{equation}
\label{e:vars_system1D}
    \frac{\partial \Var}{\partial t}  + \Vect{A}\left( \Var \right) \frac{\partial \Var}{\partial x}  = - \frac{1}{\epsilon(x)}\vect{S}(\Var),
\end{equation}
where $\Var \in \mathbb{R}^{N}$ is the unknown variable, $\Vect{A}\left( \Var \right) \in \mathbb{R}^{N \times N}$ is the system matrix of the transport term, and the term containing $\vect{S} \in \mathbb{R}^{N}$ is the possibly stiff right-hand side source term. We are particularly interested in small and spatially varying values of the relaxation time $\epsilon(x)\in \mathbb{R}^+$.

Note that \Eqn \eqref{e:vars_system1D} is a generalization of the standard form of a conservation law with right-hand side source term
\begin{equation}
\label{e:conservative_form}
    \frac{\partial \Var}{\partial t}  + \frac{\partial }{\partial x}\vect{F}\left( \Var \right)  = - \frac{1}{\epsilon(x)}\vect{S}(\Var),
\end{equation}
where $\vect{F}\left( \Var \right)$ is the flux function depending on the unknown variable. In the notation of \Eqn \eqref{e:vars_system1D}, the system matrix can then be seen as the Jacobian of the flux function, i.e., $\Vect{A} = \frac{\partial \vect{F}}{\partial \Var}$. Even for zero source term $\vect{S}= 0$, \Eqn \eqref{e:vars_system1D} is not necessarily a conservative system. In comparison to \Eqn \eqref{e:conservative_form}, the general form of \Eqn \eqref{e:vars_system1D} thus also includes so-called non-conservative systems, for which no flux function exists. Those systems occur in a lot of contexts, e.g. in rarefied gases \cite{Koellermeier2017a} and free surface flows \cite{Koellermeier2020c}.

The simplest form of the model equation \eqref{e:vars_system1D} is the scalar equation
\begin{equation}
\label{e:scalar}
    \frac{\partial w}{\partial t}  + a \frac{\partial w}{\partial x}  = - \frac{1}{\epsilon(x)} w,
\end{equation}
which models transport with constant advection velocity $a\in \mathbb{R}$ and relaxation to zero with relaxation rate $\frac{1}{\epsilon(x)}$. Note that the constant transport velocity leads to a time step constraint of the form $\Delta t \leq CFL \frac{\Delta x}{\left|a\right|}$ for given CFL number $CFL \leq 1$. For small values of $\epsilon(x)$, the right-hand sides becomes stiff, which leads to the constraint $\Delta t < \frac{1}{\epsilon(x)}$. While this system can exhibit a spectral gap in Fourier space \cite{Lafitte2010}, there is only one scalar variable, so that there will be no spectral gap in the physical space between fast and slow variables relaxing at different time scales, which characterizes many physical processes.
We therefore consider systems of equations and exemplarily consider two examples from rarefied gases in the following two sections.

\subsection{Hyperbolic moment models}
\label{sec:HME}
In rarefied gases, the mass density distribution function $f(t, x, c)$ can be expanded in a truncated Hermite sum in the microscopic velocity space $c\in \mathbb{R}$ with coefficients $f_{i}(t,x)$ for $i=3,\ldots, M$ in addition to its macroscopic moments $\rho(t,x)$, $\vel(t,x)$, $\theta(t,x)$, denoting density, bulk velocity, and temperature, respectively.
\begin{equation}
    \label{e:vars_expansion}
    f(t,x,c) = \sum_{i=0}^M f_{i}(t,x) \phi^{[\vel,\theta]}_{\vect{\alpha}}\left(\frac{c-\vel}{\theta}\right), \quad f_0 = \rho, f_1 = 0, f_2 = 1
\end{equation}
The vector of unknown variables is then given by $\Var = \left(\rho, \vel, \theta, f_3, \ldots, f_N\right) \in \mathbb{R}^{M+1}$.

The evolution of the variables is governed by the non-linear hyperbolic moment equations (HME) \cite{Cai2013b}, which are given by the system matrix $\Vect{A}_{HME} \in \mathbb{R}^{(M+1)\times (M+1)}$ defined by
\begin{equation}
\label{e:QBME_A}
\Vect{A}_{HME} = \setlength{\arraycolsep}{1pt}
\left(
  \begin{array}{cccccccc}
    \vel & \rho &  &  &   &   &   &  \\
    \frac{\theta}{\rho} & \vel & 1 &  &  &   &   &  \\
    & 2 \theta & \vel & \frac{6}{\rho} &  &  &   &   \\
    & 4 f_3 & \frac{\rho \theta}{2} & \vel & 4 &  &  &  \\
    -\frac{\theta f_3}{\rho} & 5f_4 & \frac{3f_3}{2} & \theta & \vel & 5 &  &    \\
    \vdots & \vdots & \vdots & \vdots &  \ddots & \ddots & \ddots &  \\
    -\frac{\theta f_{M-2}}{\rho} & Mf_{M-1} & \frac{\left(M-2\right)f_{M-2}+ \theta f_{M-4}}{2} \textcolor{myred}{-\frac{M(M+1)f_{M}}{2 \theta}}& -\frac{3f_{M-3}}{\rho}  &  & \theta & \vel & M \\
    -\frac{\theta f_{M-1}}{\rho} & (M\hspace{-0.1cm}+\hspace{-0.1cm}1) f_{M} & \textcolor{myred}{-f_{M-1}} \hspace{-0.1cm} + \hspace{-0.1cm} \frac{\theta f_{M-3}}{2} & \textcolor{myred}{\frac{3(M+1)f_{M}}{\rho \theta}} \hspace{-0.1cm} - \hspace{-0.1cm} \frac{3f_{M-2}}{\rho} &  &  & \theta & \vel \\
  \end{array}
\right),
\setlength{\arraycolsep}{6pt}
\end{equation}
and the source term $\vect{S}(\Var) \in \mathbb{R}^{M+1}$ on the right-hand side as the collision term that can be modelled using the simple BGK model \cite{Bhatnagar1954} as
\begin{equation}
\label{e:BGK_term}
    - \frac{1}{\epsilon(x)} \vect{S}(\Var)  = - \frac{1}{\epsilon(x)} diag \left( 0,0,0,1, \ldots,1 \right) \Var,
\end{equation}
for relaxation time $\epsilon(x) \in \mathbb{R}_+$. Note how the source term leads to a relaxation of the coefficients $f_i$ to zero, which is the state represented by equilibrium, in which the distribution function $f(t,x,c)$ is in the form of a Maxwellian and characterized by the first three moments $\rho, \vel, \theta$ alone, i.e.,
\begin{equation}
\label{e:Maxwellian}
    f_{\text{Maxwell}}(t,x,c)=\frac{\rho(t,x)}{\sqrt{2\pi\theta(t,x)}}\exp\left(
    -\frac{|c-\vel(t,x)|^2}{2\theta(t,x)} \right).
\end{equation}
For small values $\epsilon(x)$ the coefficients $f_i$ quickly relax to zero and the model is governed by the slowly evolving macroscopic variables $\rho, \vel, \theta$, clearly indicating the different scales.

\subsection{Hermite spectral model}
\label{sec:HSM}
The linearized version of the HME model in \Sect \ref{sec:HME} is called Hermite Spectral Method (HSM), see \cite{Fan2020a,Koellermeier2021}. It can be seen as a discrete velocity scheme using a spectral discretization of the velocity space corresponding to variables $f_i$. This leads to a unknown variable vector $\Var = \left( f_0, \ldots, f_M \right) \in \mathbb{R}^{M+1}$ and results in the system matrix $\Vect{A}_{HSM} \in \mathbb{R}^{(M+1)\times (M+1)}$ defined by
\begin{equation}
\label{e:grad_A_HSM}
\Vect{A}_{HSM} = \left(
  \begin{array}{ccccc}
     & 1 &   &   &   \\
     1 &  & \sqrt{2}  &   &   \\
     & \sqrt{2} &   &  \ddots &   \\
     &  & \ddots  &   & \sqrt{M}  \\
     &  &   & \sqrt{M}  &
  \end{array}
\right).
\end{equation}

The right-hand side vector $\vect{S}\left( \Var \right) \in \mathbb{R}^{M+1}$ for the 1D BGK model uses the projection onto scaled Hermite polynomials $\psi_{\alpha}$ of degree $\alpha$ and is given by
\begin{equation}
\label{e:BGK_term_HSM}
  - \frac{1}{\epsilon(x)} \vect{S}_{\alpha} = \int_{\mathbb{R}} \left( f(t,x,c) - f_{\text{Maxwell}}(t,x,c) \right) \psi_{\alpha}(c) \,dc, ~ \textrm{ for }~ \psi_{\alpha}(c) = \frac{He_{\alpha}(c)}{\sqrt{2^{\alpha} {\alpha}!}},
\end{equation}
where an analytical expression is difficult to obtain, see \cite{Fan2020a,Koellermeier2021}.

In \cite{Koellermeier2021} it was shown that the HSM model contains the same spectrum as a standard discrete velocity model, as commonly used in rarefied gases. This means that the unknown variables relax with relaxation time $\epsilon(x)$ to an equilibrium manifold on which they only evolve with respect to the macroscopic time scale given by the transport of the macroscopic variables. Thus, spectral gaps can be expected for small and/or varying values of $\epsilon(x)$.

\subsection{Spatial discretization}
In this section, we detail the spatial discretization of models of the form \eqref{e:vars_system1D}. We use the notation of polynomial viscosity matrix (PVM) methods as outlined in appendix \ref{app:NC} for standard non-conservative finite volume schemes and non-linear models. For more details on the spatial discretization, we refer to \cite{Castro2008,Pares2006}.


While numerical simulations in \Sect \ref{sec:results} are computed with the full non-linear model, a linearization is necessary to assess the spectral properties of the model and the linear stability properties of the schemes later. The PVM method \eqref{e:scheme} can then be written as
\begin{equation}\label{e:PVM-scheme_const}
    \frac{\Var_i^{n+1}-\Var_i^n}{\Delta t} = -\frac{1}{\Delta x} \left( \Vect{A} \cdot \frac{ \Var_{i+1}^n - \Var_{i-1}^n}{2} + \Vect{Q} \cdot \frac{ - \Var_{i+1}^n + 2 \Var_{i}^n - \Var_{i-1}^n}{2} \right) - \frac{1}{\epsilon(x_i)} \Vect{S} \Var_{i}^n,
\end{equation}
where the terms on the right-hand side are the non-conservative numerical flux, the numerical diffusion, and the potentially stiff source term.

Writing \Eqn \eqref{e:PVM-scheme_const} as a semi-discrete system for the unknown column vector $\vect{W} = \left( \Var_1, \Var_1, \ldots, \Var_{N_x} \right) \in \mathbb{R}^{N_x \cdot N}$ and assuming periodic boundary conditions, the system of equations reads
\begin{equation}\label{e:semi-scheme}
    \frac{\partial \vect{W}}{\partial t} = \left( - \frac{1}{2 \Delta x} \left( \Vect{A} \cdot \left(
  \begin{array}{ccccc}
     & \Vect{I} &   &   & -\Vect{I}  \\
     -\Vect{I} &  & \Vect{I}  &   &   \\
     & -\Vect{I} &   &  \ddots &   \\
     &  & \ddots  &   & \Vect{I}  \\
     \Vect{I} &  &   & -\Vect{I}  &
  \end{array}
\right) + \Vect{Q} \cdot \left(
  \begin{array}{ccccc}
     2\Vect{I} & -\Vect{I} &   &   & -\Vect{I}  \\
     -\Vect{I} & 2\Vect{I} & -\Vect{I}  &   &   \\
     & -\Vect{I} & \ddots  &  \ddots &   \\
     &  & \ddots  & \ddots  & -\Vect{I}  \\
     -\Vect{I} &  &   & -\Vect{I}  & 2\Vect{I}
  \end{array}
\right) \right) - \frac{1}{\epsilon} \Vect{S} \right) \vect{W},
\end{equation}
where the entries are block matrices containing the identity matrix $\Vect{I} \in \mathbb{R}^N$

The system \eqref{e:semi-scheme} can be simplified further to
\begin{equation}\label{e:semi-scheme_simpl}
    \frac{\partial \vect{W}}{\partial t} = \systemA \vect{W},
\end{equation}
with blockwise defined matrix
\begin{equation}\label{e:semi-scheme_simpl_matrix}
    \systemA = \left(
  \begin{array}{ccccc}
     d_0 & b &   &   & c  \\
     c & d_1 & b  &   &   \\
     & c & \ddots  &  \ddots &   \\
     &  & \ddots  & \ddots  & b  \\
     b &  &   & c  & d_{N_x}
  \end{array}
\right),
\end{equation}
that has varying diagonal entries $d_i$ and constant off diagonals $b,c$
\begin{eqnarray}
  d_i &=& -\frac{1}{\Delta x} \Vect{Q} - \frac{1}{\epsilon(x_i)} \Vect{S}, \\
  b &=& \frac{1}{2\Delta x} \left( \Vect{Q} - \Vect{A} \right), \\
  c &=& \frac{1}{2\Delta x} \left( \Vect{Q} + \Vect{A} \right).
\end{eqnarray}

From the definition of the matrix $\systemA$ in \eqref{e:semi-scheme_simpl_matrix} we can see that a large value of the entries in $\Vect{Q}$ makes the matrix more diagonally dominant and leads to a more stable scheme.
The case $\Vect{Q}=0$, which corresponds to the FCTS scheme, discretizing the transport part using simple central finite differences, would even lead to an instable scheme for moderate relaxation times.
Note that the diagonal blocks $a_i$ depend on the relaxation time $\epsilon(x_i)$ evaluated at the respective cell. In the following, we will denote $\epsilon(x_i)= \epsilon_i$.

\subsection{Spatially varying collision rates}
As an example for spatially varying collision rates, we consider two piecewise constant values and write
\begin{equation}\label{e:collision_rate_adaptive}
    \epsilon(x) = \left\{
  \begin{array}{cl}
    \epsilon_L & \textrm{if } x < 0, \\
    \epsilon_R & \textrm{if } x \geq 0, \\
  \end{array} \right.
\end{equation}
which gives rise to denoting $\Vect{W} = \left( \begin{array}{c}
    \Vect{W}_L \\
    \Vect{W}_R \\
  \end{array} \right)$ and the following decomposition of the system into two regions:
\begin{equation}\label{e:decoupled-scheme}
    \frac{\partial}{\partial t} \left( \begin{array}{c}
    \Vect{W}_L \\
    \Vect{W}_R \\
  \end{array} \right) =  \left( \begin{array}{cl}
    \systemA_{LL} & \systemA_{LR} \\
    \systemA_{RL} & \systemA_{RR} \\
  \end{array} \right) \left( \begin{array}{c}
    \Vect{W}_L \\
    \Vect{W}_R \\
  \end{array} \right),
\end{equation}
where $\Vect{W}_L$ and $\Vect{W}_R$ correspond to the values in the left part and right part of the domain, respectively.

In terms of \eqref{e:semi-scheme_simpl_matrix} the respective parts of the system are given as
\begin{equation}\label{e:semi-scheme_matrix_LL_LR}
    \systemA_{LL} = \left(
  \begin{array}{cccc}
     d_L & b &     &   \\
      c & \ddots  &  \ddots &   \\
      & \ddots  & \ddots  & b  \\
       &   & c  & d_L
  \end{array}
\right), \systemA_{LR} = \left(
  \begin{array}{cccc}
      &  &     & c  \\
      &  &     &   \\
     &  &     &   \\
     b &   &   &
  \end{array}
\right),
\end{equation}

\begin{equation}\label{e:semi-scheme_matrix_RL_RR}
    \systemA_{RL} = \left(
  \begin{array}{cccc}
      &  &     & c  \\
      &  &     &   \\
     &  &     &   \\
     b &  &     &
  \end{array}
\right), \systemA_{RR} = \left(
  \begin{array}{cccc}
     d_R & b &    &   \\
     c & \ddots  &  \ddots &   \\
      & \ddots  & \ddots  & b  \\
       &   & c  & d_R
  \end{array}
\right),
\end{equation}
with
\begin{eqnarray}
  d_L &=& -\frac{1}{\Delta x} \Vect{Q} - \frac{1}{\epsilon_L} \Vect{S}, \\
  d_R &=& -\frac{1}{\Delta x} \Vect{Q} - \frac{1}{\epsilon_R} \Vect{S}.
\end{eqnarray}
Note that the only difference in the two block matrices on the diagonal is the value of the source term, which uses either $\epsilon_L$ for $\systemA_{LL}$ or $\epsilon_R$ for $\systemA_{RR}$, respectively. The sparse off-diagonal blocks contain the information of the boundary conditions coupling the left and right part of the domain.\\

The coupled system can be written as two subsystems
\begin{eqnarray}
    \frac{\partial}{\partial t} \Vect{W}_L &=& \systemA_{LL} \cdot \Vect{W}_L + \systemA_{LR} \cdot \Vect{W}_R \label{e:decoupled-systemsL}, \\
    \frac{\partial}{\partial t} \Vect{W}_R &=& \systemA_{RL} \cdot \Vect{W}_L + \systemA_{RR} \cdot \Vect{W}_R \label{e:decoupled-systemsR}.
\end{eqnarray}

Without loss of generality, we assume $\epsilon_L \ll \epsilon_R$ and call the left subsystem \eqref{e:decoupled-systemsL} the stiff system, whereas the right subsystem \eqref{e:decoupled-systemsR} is non-stiff. After choosing a specific spatial discretization, i.e., the PVM method, a stability analysis can be performed.

\begin{remark}
  Note that the distinction between stiff and non-stiff region made by setting the collision rate in \eqref{e:collision_rate_adaptive} can also be much more arbitrary. We can have several, distinct regions with small collision rates $\epsilon_{L}$ or large collision rates $\epsilon_{R}$. A permutation of the variable vector $\Vect{W}$ can be performed to split the system to an upper and lower part with small or large collision rate, respectively. This will lead to more entries in the off-diagonal blocks $\systemA_{RL}, \systemA_{LR}$ and a larger bandwidth in the diagonal blocks $\systemA_{LL}, \systemA_{RR}$. But the analysis can be performed in the same way.
\end{remark}

\subsection{Spectral analysis}
    In this section, we give a concise statement about the spectral properties of the system matrix $\systemA$ \eqref{e:semi-scheme_simpl_matrix} for spatially varying relaxation times. In addition to the assumption of a linearized system, i.e. $\Vect{A} = const$, we use the following assumptions throughout this section

    \begin{itemize}
      \item[(A1)] $\Vect{A}$ is symmetric. This is true for the HSM model \eqref{e:grad_A_HSM}. For the HME model \eqref{e:QBME_A} it requires a linearization around equilibrium and a proper symmetrization \cite{Cai2013b}.
      \item[(A2)] the source term $- \frac{1}{\epsilon_i} \Vect{S}$ is given by a diagonal matrix $- \frac{1}{\epsilon_i} \widetilde{\Vect{I}} := - \frac{1}{\epsilon_i} diag(0,0,0,1,\ldots,1)$, modeling the conservation of mass, momentum and energy and the relaxation of higher order moments. This is true for the BGK operator of the HME model \eqref{e:BGK_term}. For the HSM model \eqref{e:BGK_term_HSM}, it requires a previous redefinition of the variable space \cite{Koellermeier2021}.
    \end{itemize}

    As common for the stability analysis of numerical schemes, the analysis assumes small deviations from some linearized state. Linear stability is then a necessary condition for fully non-linear simulations using the numerical scheme. 
    Note that despite the linearization and the assumptions in this section, numerical simulations of initial value problems for \eqref{e:semi-scheme_simpl} are typically performed using the full non-linear model. 
    The numerical results in \Sect \ref{sec:results} show that the linear stability results can readily be applied to the fully non-linear system as well.

    The main result of this section is the following theorem on the general characterization of the eigenvalues, which is later specified for different spatial discretizations (Upwind, Lax-Friedrichs, FORCE).

    \begin{theorem}
        \label{th:spectrum}
        Under the assumptions (A1),(A2) the spectrum $\sigma\left(\systemA\right)$ for the models described in \Sects \ref{sec:HME} and \ref{sec:HSM} consists of one slow cluster and remaining fast cluster(s) depending on the values of the relaxation time $\epsilon$ evaluated on the grid as

        \begin{equation}\label{e:sigma}
            \sigma\left(\systemA\right) \in C\left(\lambda_s,R\right) \cup \left( \underset{i}{\bigcup} C\left(\lambda_{\epsilon_i},R\right) \right),
        \end{equation}
        with circles $C(\lambda,R)$ in the complex plane centered around $\lambda$ with radius $R$.

        The values $\lambda_s,\lambda_{\epsilon_i},R$ depend on the spatial discretization scheme, i.e., the definition of the PVM matrix $\Vect{Q}$ and the relaxation time as follows
        \begin{eqnarray*}
          \lambda_s &=& - \frac{1}{\Delta x} \lambda\left(\Vect{Q}\right),\\
          \lambda_{\epsilon_i} &=& - \frac{1}{\Delta x} \lambda\left(\Vect{Q}\right) - \frac{1}{\epsilon_i},\\
          R &=& \frac{1}{2 \Delta x} \left( \left|\lambda_{max}\left(\Vect{Q}-\Vect{A}\right) \right| + \left|\lambda_{max}\left(\Vect{Q}+\Vect{A}\right) \right| \right).
        \end{eqnarray*}
    \end{theorem}
    \begin{proof}
        According to the Gershgorin circle theorem \cite{Tretter2008}, all eigenvalues of a block-wise defined matrix $\systemA$ of the type in \Eqn \eqref{e:semi-scheme_simpl_matrix} are included in the following domains

        \begin{equation}\label{e:Gershgorin}
            G_i = \sigma(a_i) \cup \left( \underset{i}{\bigcup} C\left(\lambda_{d_i}, \| b \| + \| c \|\right) \right),
        \end{equation}
        where $a_i,b_i,c_i$ are the block matrices in the $i$-th row of the matrix $\systemA$. The norm $\| \cdot \|$ is the spectral norm which evaluates to the absolute value of the maximum eigenvalue, for the symmetric matrices used here.

        For the computation of the respective eigenvalues, we make use of the fact that for any holomorphic function $f$, the eigenvalues $\lambda(f(\Vect{A}))$ are simply given by $f(\lambda(\Vect{A}))$. This will be especially useful for the PVM methods, for which the viscosity matrix $\Vect{Q}$ is a function of $\Vect{A}$.

        The eigenvalues of $a_i = -\frac{1}{\Delta x} \Vect{Q} - \frac{1}{\epsilon_i} \Vect{S}$ cannot be obtained without prior knowledge of either $\Vect{Q}$ or $\Vect{S}$. However, they have been computed for several explicit moment models in \cite{Lafitte2010}. Using only assumption (A2) about the form of the relaxation matrix $\Vect{S}$, see also \Eqn \eqref{e:BGK_term}, a subset of three slow eigenvalues is given as $\lambda_{1,2,3}$. Due to the form of $\Vect{S}$, see also \Eqn \eqref{e:BGK_term}, those eigenvalues are among the eigenvalues of $\lambda_{s} = -\frac{1}{\Delta x} \lambda(\Vect{Q})$ and the remaining eigenvalues are given by $\lambda_{\epsilon_i} = -\frac{1}{\Delta x} \lambda(\Vect{Q}) - \frac{1}{\epsilon_i}$.

        The maximum eigenvalue of $b$ can be computed using $\| b \| = \frac{1}{2\Delta x} \left| \lambda_{max}\left( \Vect{Q} - \Vect{A} \right) \right|$ and the insertion of the viscosity matrix $\Vect{Q}$ as a function of $\Vect{A}$. The computation of $\| c \| = \frac{1}{2\Delta x} \left| \lambda_{max}\left( \Vect{Q} + \Vect{A} \right) \right|$ follows in the same way.
    \end{proof}

    Theorem \ref{th:spectrum} states that there is one main slow cluster and several fast clusters that are determined by the relaxation times. This result is similar to the case of discrete velocity models in \cite{Lafitte2017,Lafitte2010,Melis2019}. However, the simple and explicit form of the model allows for a straightforward characterization of the spectrum without moving to Fourier space first. This will be clear in the following sections, where we apply Theorem \ref{th:spectrum} to the Upwind, Lax-Friedrichs, and FORCE scheme.

\subsubsection{Upwind scheme}
    The Upwind scheme uses minimal viscosity $\Vect{Q} = \left|\Vect{A}\right|$ \eqref{e:upwind_PVM}, such that the values $\lambda_s,\lambda_{\epsilon_i},R$ in Theorem \ref{th:spectrum} are computed
    \begin{equation}\label{e:upwind_lambda_s}
        \lambda_s =  - \frac{1}{\Delta x} \lambda\left(\Vect{Q}\right) = - \frac{\left|\lambda\left(\Vect{A}\right)\right|}{\Delta x},
    \end{equation}
    \begin{equation}\label{e:upwind_lambda_i}
        \lambda_i =  - \frac{1}{\Delta x} \lambda\left(\Vect{Q}\right) - \frac{1}{\epsilon_i} = - \frac{\left|\lambda\left(\Vect{A}\right)\right|}{\Delta x}  - \frac{1}{\epsilon_i},
    \end{equation}
    \begin{eqnarray}
      R &=& \frac{1}{2 \Delta x} \left( \left|\lambda_{max}\left(\Vect{Q}-\Vect{A}\right) \right| + \left|\lambda_{max}\left(\Vect{Q}+\Vect{A}\right) \right| \right) = \frac{1}{2 \Delta x} \left( \left|\lambda_{max}\left(\left|\Vect{A}\right|-\Vect{A}\right) \right| + \left|\lambda_{max}\left(\left|\Vect{A}\right|+\Vect{A}\right) \right| \right) \nonumber \\
       &=& \frac{1}{2 \Delta x} \left( 2 \left|\lambda_{min}\left(\Vect{A}\right) \right| + 2 \left|\lambda_{max}\left(\Vect{A}\right) \right| \right) \leq \frac{1}{2 \Delta x} 2 \left|\lambda_{max}\left(\Vect{A}\right) \right| = \frac{\lambda_{max}\left(\Vect{A}\right)}{\Delta x}.
    \end{eqnarray}
    for positive $\lambda_{max}\left(\Vect{A}\right)>0$.

    For the upwind scheme, the slow cluster centered at $\lambda_s$ depends on the propagation speed of the model and the spatial grid, while the fast clusters $\lambda_i$ include also the relaxation of faster values with the respective relaxation time evaluated at the grid. The radius $R$ of the clusters depends on the maximal eigenvalues and the grid.

\subsubsection{Lax-Friedrichs scheme}
    For the Lax-Friedrichs scheme, the values $\lambda_s,\lambda_{\epsilon_i},R$ in Theorem \ref{th:spectrum} are computed using $\Vect{Q} = \frac{\Delta x}{\Delta t} \Vect{I}$ \eqref{e:LF_PVM} and the time step size is given by a macroscopic CFL condition $\Delta t = CFL \cdot \frac{\Delta x}{\lambda_{max}\left(\Vect{A}\right)}$ as follows

    \begin{equation}\label{e:LF_lambda_s}
        \lambda_s =  - \frac{1}{\Delta t} = - \frac{\lambda_{max}\left(\Vect{A}\right)}{CFL \cdot \Delta x},
    \end{equation}
    \begin{equation}\label{e:LF_lambda_i}
        \lambda_i = - \frac{1}{\Delta t} - \frac{1}{\epsilon_i} = - \frac{\lambda_{max}\left(\Vect{A}\right)}{CFL \cdot \Delta x} - \frac{1}{\epsilon_i},
    \end{equation}
    \begin{eqnarray}
      R &=& \frac{1}{2 \Delta x} \left( \left|\lambda_{max}\left(\frac{\Delta x}{\Delta t} \Vect{I}-\Vect{A}\right) \right| + \left|\lambda_{max}\left(\frac{\Delta x}{\Delta t} \Vect{I}+\Vect{A}\right) \right| \right) \\
        &=& \frac{1}{2 \Delta x} \left( \left| \frac{\lambda_{max}\left(\Vect{A}\right)}{CFL} - \lambda_{max}\left(\Vect{A}\right) \right| + \left| \frac{\lambda_{max}\left(\Vect{A}\right)}{CFL} + \lambda_{max}\left(\Vect{A}\right) \right| \right)\\
       &=& \frac{1}{2 \Delta x} \left( \left(\frac{1}{CFL} - 1\right) \lambda_{max}\left(\Vect{A}\right) + \left(\frac{1}{CFL} + 1\right) \lambda_{max}\left(\Vect{A}\right) \right) = \frac{\lambda_{max}\left(\Vect{A}\right)}{CFL \cdot \Delta x},
    \end{eqnarray}
    for positive $\lambda_{max}\left(\Vect{A}\right)>0$.

    The Lax-Friedrichs scheme shows similar dependence of $\lambda_s, \lambda_i, R$ on the model parameters compared to the upwind scheme. However, the $CFL$ number enters in the denominator, which has an important effect on the stability of the scheme later.

\subsubsection{FORCE scheme}
    For the FORCE scheme, the values $\lambda_s,\lambda_{\epsilon_i},R$ in Theorem \ref{th:spectrum} can be computed using $\Vect{Q} = \frac{\Delta x}{2\Delta t} \Vect{I} + \frac{\Delta t}{2\Delta x} \Vect{A}^2$ \eqref{e:FORCE_PVM} and the time step size is given by a macroscopic CFL condition $\Delta t = CFL \cdot \frac{\Delta x}{\lambda_{max}\left(\Vect{A}\right)}$ as follows

    \begin{equation}\label{e:FORCE_lambda_s}
        \lambda_s =  - \frac{1}{2\Delta t} - \frac{\Delta t}{2\Delta x} \lambda\left(\Vect{A}\right)^2 \geq - \frac{\lambda_{max}\left(\Vect{A}\right)}{2 \Delta x} \left(\frac{1}{CFL} + CFL\right),
    \end{equation}
    \begin{equation}\label{e:FORCE_lambda_i}
        \lambda_i = - \frac{1}{2\Delta t} - \frac{\Delta t}{2\Delta x} \lambda\left(\Vect{A}\right)^2 - \frac{1}{\epsilon_i} \geq - \frac{\lambda_{max}\left(\Vect{A}\right)}{2 \Delta x} \left(\frac{1}{CFL} + CFL\right) - \frac{1}{\epsilon_i},
    \end{equation}
    \begin{eqnarray}
      R &=& \frac{1}{2 \Delta x} \left( \left|\lambda_{max}\left(\frac{\Delta x}{2\Delta t} \Vect{I} + \frac{\Delta t}{2\Delta x} \Vect{A}^2 - \Vect{A}\right) \right| + \left|\lambda_{max}\left(\frac{\Delta x}{2\Delta t} \Vect{I} + \frac{\Delta t}{2\Delta x} \Vect{A}^2 + \Vect{A}\right) \right| \right) \nonumber \\
       &=& \frac{1}{2 \Delta x} \left( \left(\frac{1}{2 CFL} + \frac{CFL}{2} - 1\right) \lambda_{max}\left(\Vect{A}\right) + \left(\frac{1}{2 CFL} + \frac{CFL}{2} + 1\right) \lambda_{max}\left(\Vect{A}\right) \right) \nonumber \\
       &=& \frac{\lambda_{max}\left(\Vect{A}\right)}{2 \Delta x} \left(\frac{1}{CFL} + CFL\right).
    \end{eqnarray}
    for positive $\lambda_{max}\left(\Vect{A}\right)>0$.

    The FORCE scheme results in slightly more complicated formulas to compute $\lambda_s, \lambda_i, R$. This is mainly due to the definition of the polynomial viscosity matrix \eqref{e:FORCE_PVM}. Note that the FORCE scheme yields the same results as the upwind scheme for $CFL=1$.

\subsubsection{Numerical validation of spatial discretizations' spectral properties}
    In \Fig \ref{fig:Spectrum} the results of Theorem \ref{th:spectrum} are validated using numerical values of the actual spectrum. For the numerical computation, we used the HME model \eqref{e:QBME_A} with $M+1=5$ equations, linearized around equilibrium with $(\rho,u,\theta)=(1,\pi,1)$, such that the maximum eigenvalue evaluates to $\lambda_{max}\approx 6$. The spatial discretization is performed on the grid $[-1,1]$ with $\Delta x = 1/50$, i.e. 100 cells. A spatially varying relaxation time is chosen according to piecewise constant values \eqref{e:collision_rate_adaptive}
    \begin{equation}\label{e:spatially_varying_eps}
        \epsilon(x) = \left\{
        \begin{array}{cl}
            \epsilon_L = 10^{-4}& \textrm{if } x < 0, \\
            \epsilon_R = 10^{-3}& \textrm{if } x \geq 0. \\
        \end{array} \right.
    \end{equation}
    The $CFL$ number is set to $0.75$.

    In all three cases, we can see that the regions proposed in Theorem \ref{th:spectrum} correctly include all eigenvalues of the scheme. This includes the split into three regions. The slow cluster contains the macroscopic evolution governed by the flow speeds $\lambda_{max}$. The second cluster is a fast cluster governed by the relaxation time $\epsilon_R$. The third and fastest cluster is governed by the relaxation time $\epsilon_L$. On the one hand, the methods slightly differ in the value of $R$, the radius of the three clusters, which is influenced by the $CFL$ number. We obtain $R_{upwind} < R_{FORCE} < R_{LF}$ (note the different scaling of the $y$-axis). On the other hand, also the position of the clusters is slightly different with a similar relation $\lambda_{upwind} > \lambda_{FORCE} > \lambda_{LF}$. Large negative real parts of the eigenvalues result in a decay in time. This means that the upwind scheme is the least diffusive while the Lax-Friedrichs scheme introduces a lot of diffusion.

    \begin{figure}[htb!]
        \centering
        \begin{subfigures}
        \subfloat[Upwind scheme. \label{fig:spectrum_upwind}
        ]{\includegraphics[width=1\linewidth]{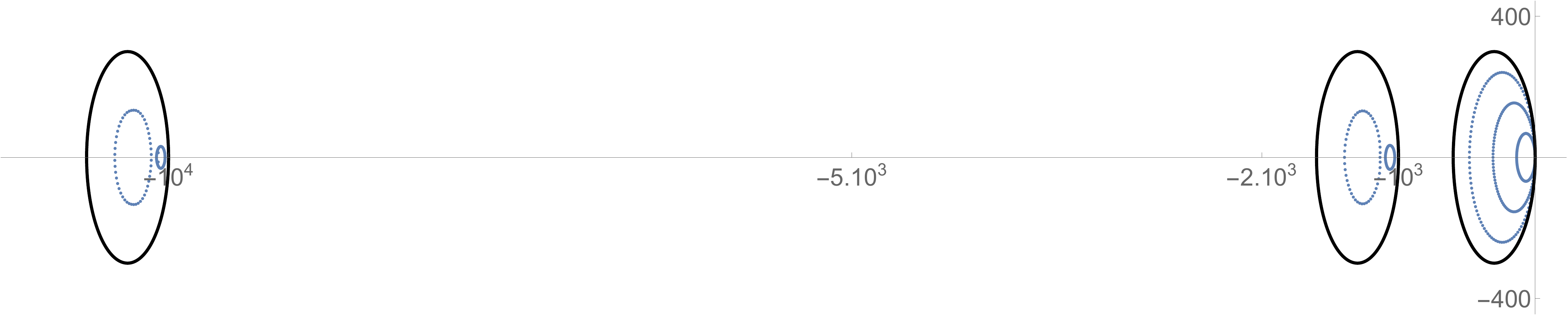}}\\
        \subfloat[Lax-Friedrichs scheme. \label{fig:spectrum_LF}
        ]{\includegraphics[width=1\linewidth]{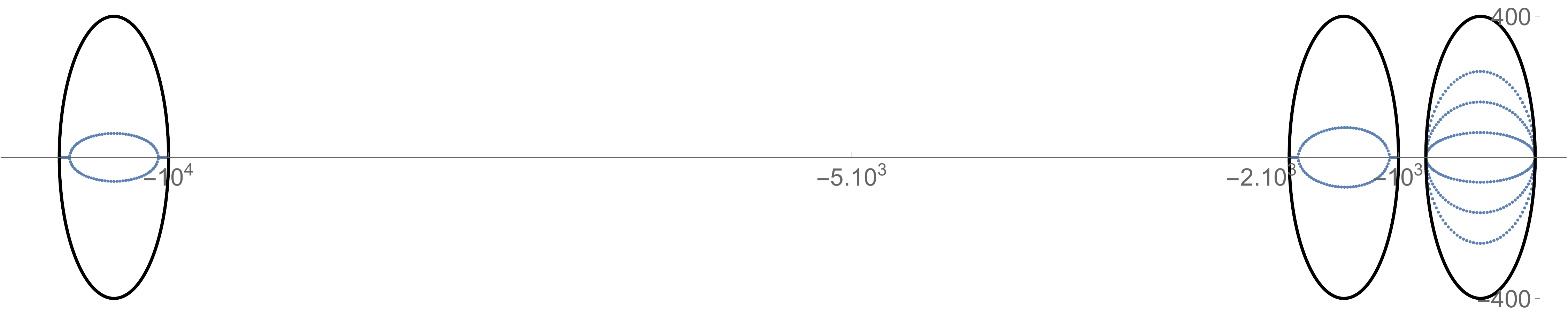}}\\
        \subfloat[FORCE scheme. \label{fig:spectrum_FORCE}
        ]{\includegraphics[width=1\linewidth]{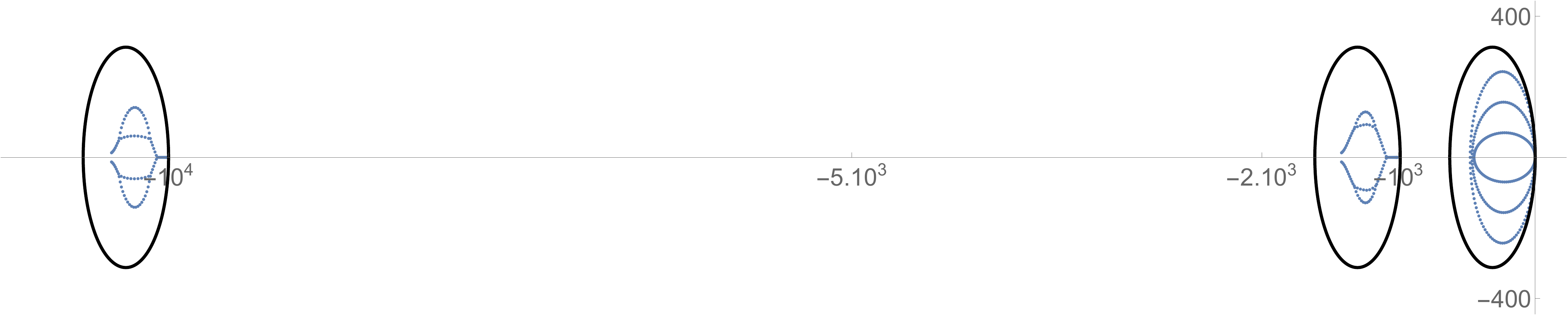}}
        \end{subfigures}
        \caption{Numerical eigenvalues and analytical spectra according to theorem \ref{th:spectrum} match for the different schemes: Upwind (top), LF (middle) and FORCE (bottom). Numerical eigenvalues are plotted in blue, Gershgorin circles from Theorem \ref{th:spectrum} are drawn in black. The model is HME \eqref{e:QBME_A} with $M+1=5$ equations, linearized around equilibrium with $(\rho,u,\theta)=(1,\pi,1)$, i.e. $\lambda_{max}\approx 6$, $\epsilon_L=10^{-4},\epsilon_L=10^{-3}$, $\Delta x = 1/50$, $CFL = 0.75$. }
        \label{fig:Spectrum}
    \end{figure}     
\section{Standard time integration schemes}
\label{sec:standard_schemes}
After a detailed investigation of the spectral properties of the semi-discrete model, we now investigate how the model can be integrated in time in a stable way, which is the main focus of this paper. First we consider standard time integration schemes, before deriving new and more suitable adaptive schemes. We consider the general setup with spatially varying but piecewise constant relaxation times according to \Eqn \eqref{e:spatially_varying_eps}.

As standard time integration schemes we consider all schemes that cannot take into account the spatial variation of the relaxation time. This typically leads to a severe time step constraint, as we will show throughout this section for the simple forward Euler scheme (FE). The Projective Forward Euler scheme (PFE) already mitigates time step constraint of the fastest eigenvalue cluster, but does not benefit from potentially slower eigenvalues in other parts of the domain.

\subsection{Forward Euler scheme (FE)}
\label{sec:FE}
The simple forward Euler scheme performs one explicit time step using a time step size $\Delta t$, as outlined in \Fig \ref{fig:FE_grid}. The update is given by
\begin{figure}[htb!]
    \centering
    \includegraphics[width=0.75\textwidth]{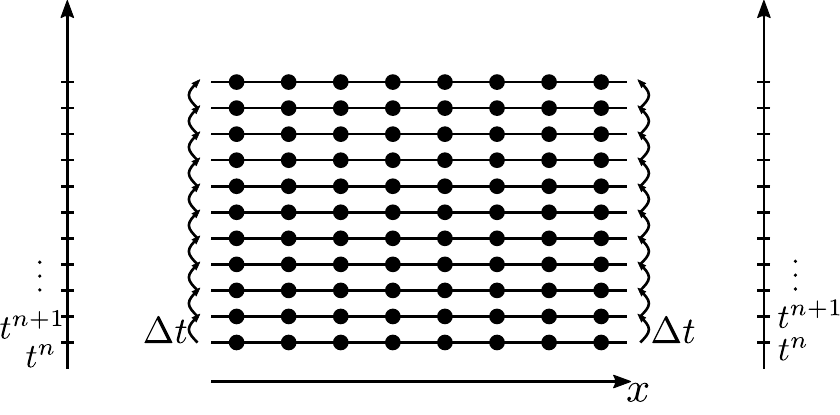}
    \caption{Forward Euler scheme (FE) with small time step $\Delta t = \mathcal{O}(\epsilon)$ in the whole domain.}
    \label{fig:FE_grid}
\end{figure}

\begin{equation}\label{e:FE_update}
    \Vect{W}^{n+1} = \Vect{W}^{n} + \Delta t \systemA \Vect{W}^{n} = \left(\Vect{I} + \Delta t \systemA\right) \Vect{W}^{n}
\end{equation}
where the matrix $\transitionA_{FE} = \Vect{I} + \Delta t \systemA$ is the so-called transition matrix, that describes the transition from the current values $\Vect{W}^{n}$ to $\Vect{W}^{n+1}$.


The stability domain of the FE scheme, based on the model equation $\partial_t w = \lambda w$, with $\lambda \in \mathbb{C}^-$ is shown in \Fig \ref{fig:FE_stab_region} and given by
\begin{equation}\label{e:FE_stab_region}
    \lambda \in C\left( - \frac{1}{\Delta t}, \frac{1}{\Delta t} \right)
\end{equation}

\begin{figure}[htb!]
        \centering
        \begin{subfigures}
        \subfloat[Stability region of Forward Euler scheme (FE). The whole domain uses one time step $\Delta t$ that determines the stability of all modes. \label{fig:FE_stab_region}
        ]{\includegraphics[width=0.7\linewidth]{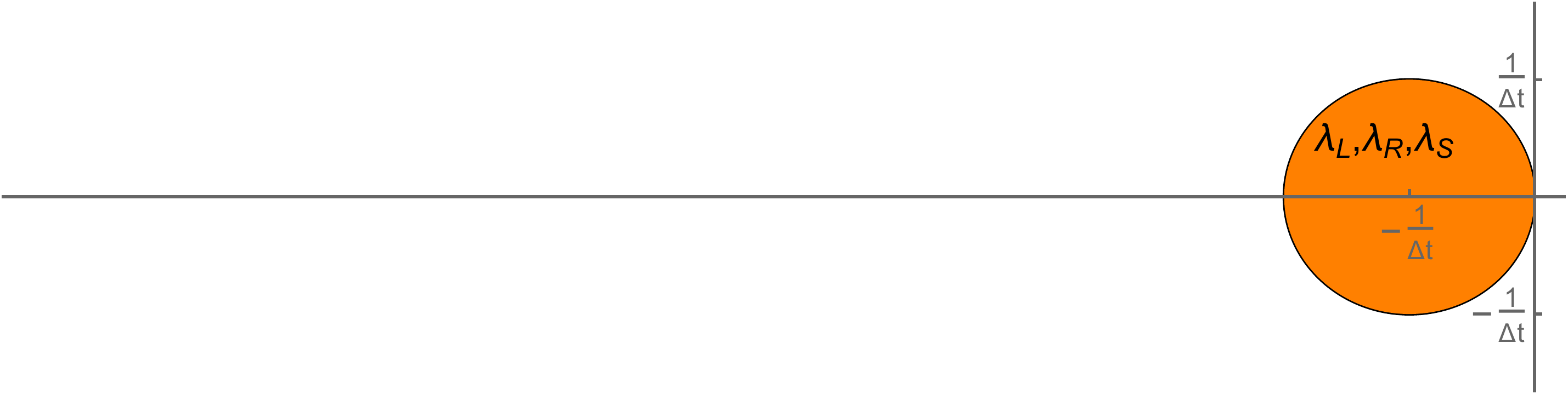}}\\
        \subfloat[Stability region of Projective Forward Euler scheme (PFE). The whole domain uses one inner time step $\delta t$ for fast modes and one time step $\Delta t$ for the other modes. \label{fig:PFE_stab_region}
        ]{\includegraphics[width=0.7\linewidth]{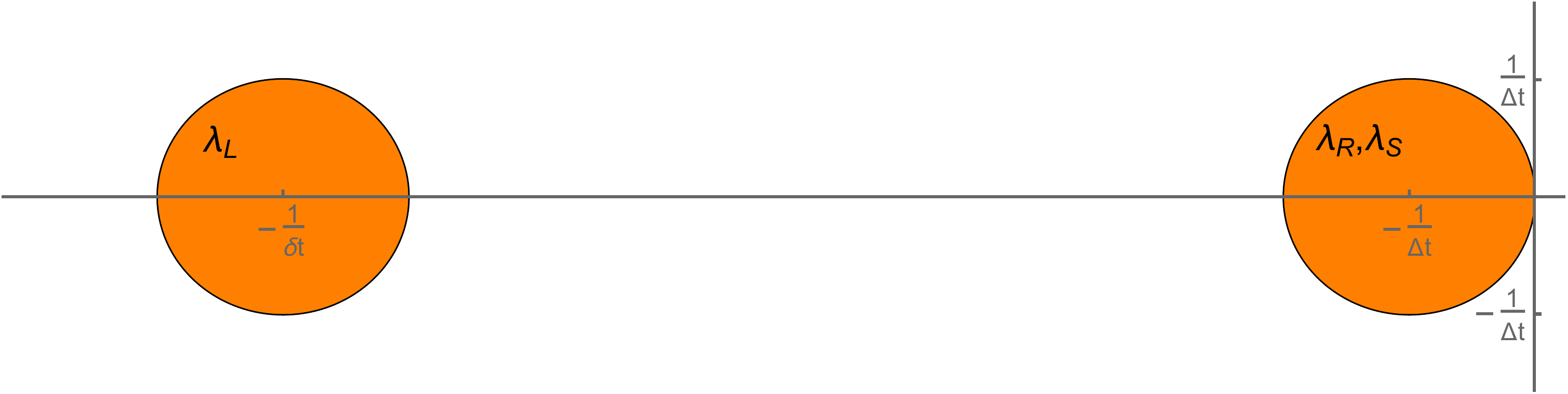}}
        \end{subfigures}
        \caption{Stability regions of FE and PFE time integration schemes. The respective eigenvalues need to be located within the specified domains for stability of the scheme. }
        \label{fig:stab_regions}
\end{figure}

The spectral analysis reveals the respective bounds on the time step size, depending on the spatial discretization, the $CFL$ number, and the relaxation times $\epsilon_{L} \ll \epsilon_R$ from \Eqn \eqref{e:spatially_varying_eps}.
Including the whole spectrum for the model analyzed in Theorem \ref{th:spectrum} within the stability domain of the FE scheme, we obtain the following stability condition
\begin{equation}\label{e:FE_stab_cond}
    \frac{1}{\Delta t} \geq \frac{1}{2} \left( |\lambda_{\epsilon}| + R \right)
\end{equation}

Inserting $\Delta t = CFL \frac{\Delta x}{\lambda_{max}}$ and known values of $\lambda_{\epsilon}$ and $R$ for the different schemes, yields the following stability conditions:
\begin{itemize}
  \item[1.] the upwind scheme is conditionally stable for $CFL \leq \frac{2\epsilon_L \lambda_{max}}{\Delta x + 2\epsilon_L \lambda_{max}}$.
  \item[2.] the Lax-Friedrichs scheme is unconditionally unstable for all $CFL$ values.
  \item[3.] the FORCE scheme is conditionally stable for $CFL \leq -\frac{\Delta x}{2 \epsilon_L \lambda_{max}} + \sqrt{\frac{\Delta x}{2 \epsilon_L \lambda_{max}}^2+1}$.
\end{itemize}
We conclude that both the upwind and the FORCE scheme are only stable under a very small $CFL$ number that is of the order $\mathcal{O}(\epsilon)$, while the Lax-Friedrichs scheme is unconditionally unstable and cannot be stabilized even by a small $CFL$ number. The severe time step constraint for upwind and FORCE is prohibitive in many applications and more suitable methods needs to be used.

Plotting the eigenvalues of the transition matrix with an upwind discretization in \Fig \ref{fig:FE_transition_stab_region} shows that the method is indeed stable and the stability bounds are relatively sharp as larger values of $\Delta t$ or $CFL$, respectively, would lead to an unstable scheme.
\begin{figure}[htb!]
        \centering
        \begin{subfigures}
        \subfloat[Forward Euler. \label{fig:FE_transition_stab_region}
        ]{\includegraphics[width=0.46\linewidth]{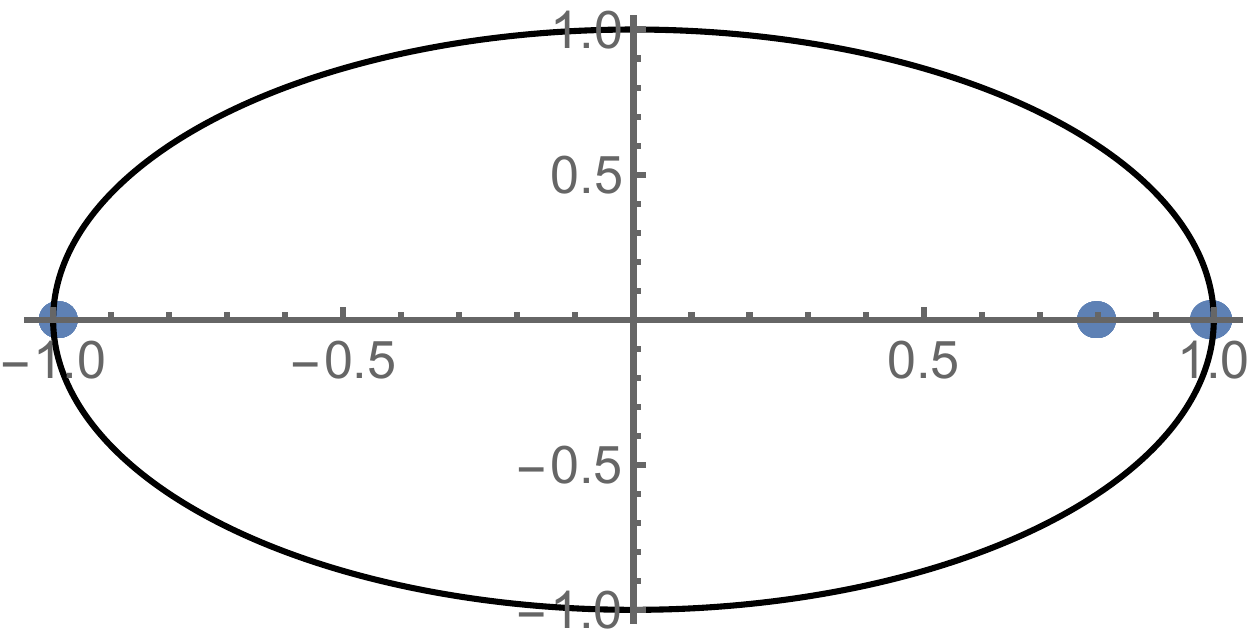}}
        \subfloat[Projective Forward Euler. \label{fig:PFE_transition_stab_region}
        ]{\includegraphics[width=0.46\linewidth]{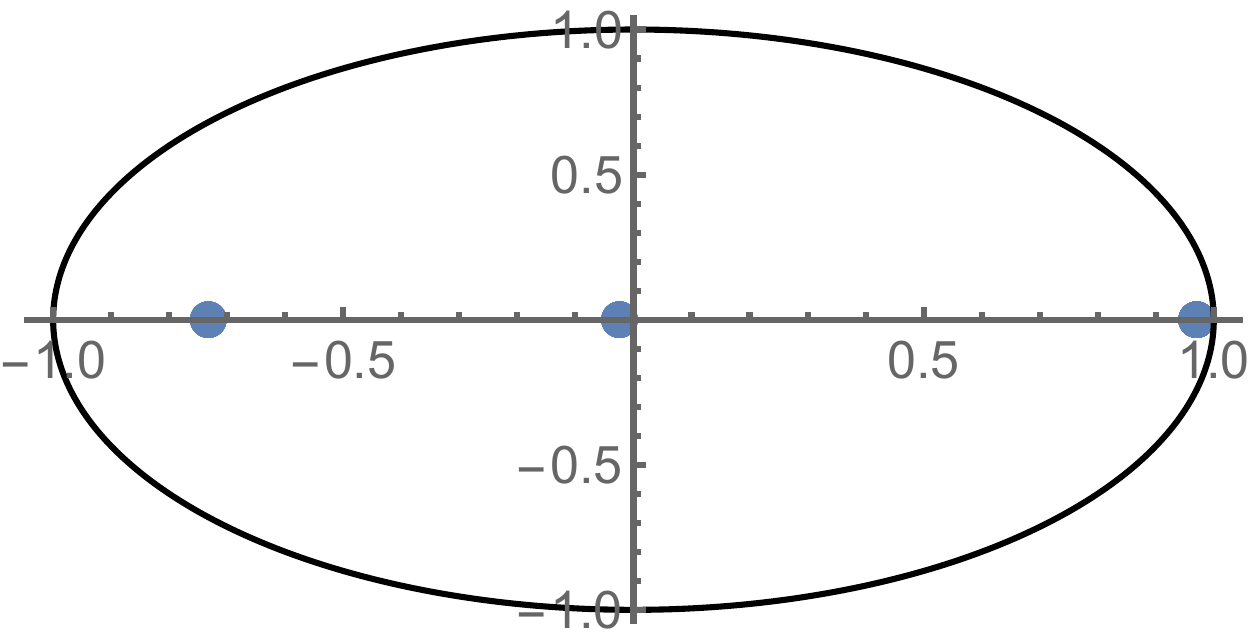}}
        \end{subfigures}
        \caption{Numerical spectrum of the transition matrix $\transitionA$ for Forward Euler (left) and Projective Forward Euler (right). Both schemes are stable if parameters are chosen according to the derived analytical values, while the estimates are relatively sharp as eigenvalues are close to stability boundary. Upwind spatial discretization, $(\rho,u,\theta)=(1,\pi,1)$, i.e. $\lambda_{max}\approx 6$, $\epsilon_L=10^{-4},\epsilon_L=10^{-3}$, $\Delta x = 1/10$. }
        \label{fig:transition_stab_regions}
\end{figure}

\subsection{Projective Forward Euler scheme (PFE)}
\label{sec:PFE}
The Projective Forward Euler scheme (PFE) is an explicit, asymptotic-preserving scheme that combines $K+1$ small time steps of size $\delta t$ with an extrapolation step over the remaining $\Delta t - (K+1) \delta t$ to achieve the value at the next time step, as outlined in \Fig \ref{fig:PFE_grid}
\begin{figure}[htb!]
    \centering
    \includegraphics[width=0.75\textwidth]{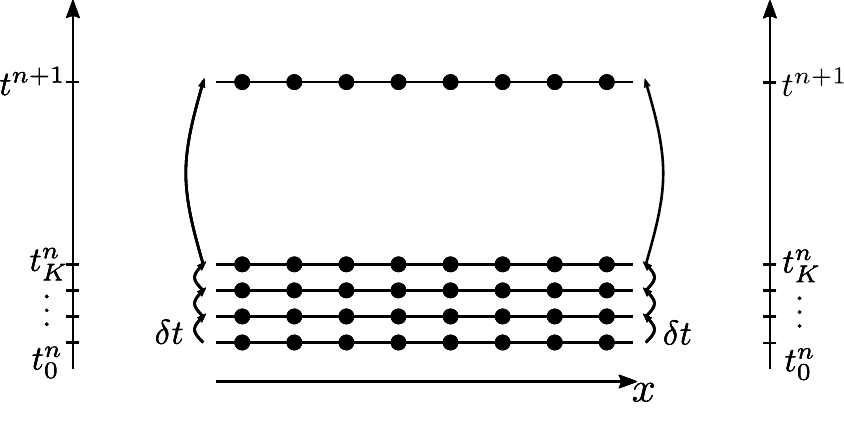}
    \caption{Projective Forward Euler scheme (PFE) with $K=2$ and small time step $\delta t = \mathcal{O}(\epsilon)$ in the whole domain.}
    \label{fig:PFE_grid}
\end{figure}

The update is computed as follows
\begin{eqnarray}
    \Vect{W}^{n,k+1} &=& \Vect{W}^{n,k} + \delta t \systemA \cdot \Vect{W}^{n,k} , k = 0,\ldots,K\\
    \Vect{W}^{n+1} &=& \Vect{W}^{n,K+1} + \left(\Delta t - (K+1) \delta t\right) \frac{\Vect{W}^{n,K+1} - \Vect{W}^{n,K}}{\delta t},
\end{eqnarray}

The stability domain of the PFE scheme, again based on the model equation $\partial_t w = \lambda w$, with $\lambda \in \mathbb{C}^-$ is shown in \Fig \ref{fig:PFE_stab_region} and given by
\begin{equation}\label{e:PFE_stab_region}
    \lambda \in C\left( - \frac{1}{\Delta t}, \frac{1}{\Delta t} \right) \cup C\left( - \frac{1}{\delta t}, \frac{1}{\delta t}\left(\frac{\delta t}{\delta t}\right)^{\frac{1}{K}} \right)
\end{equation}

Using the spectral analysis of the previous section, we can again derive the respective bounds on the parameters $\delta t$,$\Delta t$, and $K$ depending on the spatial discretization, the $CFL$ number, and the relaxation times $\epsilon_{L,R}$.
In order to include the whole spectrum for the model analyzed in Theorem \ref{th:spectrum} within the stability domain of the PFE scheme, we consider the constant relaxation time case $\epsilon = const$. We then determine the parameters based on \Eqn \eqref{e:PFE_stab_region} and Theorem \ref{th:spectrum} as
\begin{eqnarray}
  \frac{1}{\Delta t} &=& -\lambda_{s} \label{e:PFE_stab_cond1}\\
  \frac{1}{\delta t} &=& -\lambda_{\epsilon} \label{e:PFE_stab_cond2}\\
  \frac{1}{\delta t}\left(\frac{\delta t}{\Delta t}\right)^{\frac{1}{K}} &\geq& R\label{e:PFE_stab_cond3}
\end{eqnarray}

Inserting $\Delta t = CFL \frac{\Delta x}{\lambda_{max}}$ and known values of $\lambda_{\epsilon}$ and $R$ for the different schemes, yields:
\begin{itemize}
  \item[1.] the upwind scheme is conditionally stable for $\delta t = \frac{1}{\frac{\lambda_{max}}{\Delta x} + \frac{1}{\epsilon}} = \mathcal{O}(\epsilon)$, $K=1$, and  $CFL \leq 1$.
  \item[2.] the Lax-Friedrichs scheme is conditionally stable for $\delta t = \frac{1}{\frac{\lambda_{max}}{CFL \Delta x} + \frac{1}{\epsilon}} = \mathcal{O}(\epsilon)$, $K=1$, and  $CFL \leq 1$.
  \item[3.] the FORCE scheme is conditionally stable for $\delta t = \frac{1}{\frac{\lambda_{max}}{\Delta x}\left(\frac{1}{CFL} + CFL\right)  + \frac{1}{\epsilon}} = \mathcal{O}(\epsilon)$, $K=1$, and  $CFL \leq 1$.
\end{itemize}

Interestingly, the Lax-Friedrichs scheme is stable in comparison to the FE scheme. Note that the value $K=1$ is chosen here for convenience. Other values are possible and extend the stability region towards the slow cluster, see \cite{Melis2019}.

The eigenvalues of the transition matrix $\transitionA_{PFE}$ with an upwind spatial discretization and parameters according to the aforementioned stability conditions are plotted in \Fig \ref{fig:PFE_transition_stab_region}. Again, all eigenvalues are inside the unit circle and we conclude that the method is indeed stable for the parameter settings predicted by our analysis. The eigenvalues $\lambda_i$ are close to the stability boundary $\|\lambda_i \| < 1$, which indicates that both the estimates of the spectrum of the model equation and the stability properties of the scheme are relatively sharp.

The PFE scheme overcomes the restrictive time step constraint of the FE scheme in case of small relaxation times. It does not, however, make use of potential spatially varying relaxation times. If the relaxation time is only small in some parts of the domain, an adaptive method needs to be chosen for larger speedup, which will be explained in the next section.     
\section{Spatially adaptive time integration schemes }
\label{sec:adaptive_schemes}

We now need to construct time-stepping methods with matching stability region. Therefore, we consider a special treatment of the stiff and non-stiff parts of the domain.

More precisely, we consider the transition from time step $n$ to time step $n+1$ and write the update as
\begin{equation}\label{decoupled-scheme}
        \left( \begin{array}{c}
        \Vect{W}_L^{n+1} \\
        \Vect{W}_R^{n+1} \\
      \end{array} \right) =
      \left( \begin{array}{cc}
        \transitionA_{LL}^{scheme} & \transitionA_{LR}^{scheme} \\
        \transitionA_{RL}^{scheme} & \transitionA_{RR}^{scheme} \\
      \end{array} \right)
       \left( \begin{array}{c}
        \Vect{W}_L^{n} \\
        \Vect{W}_R^{n} \\
      \end{array} \right),
\end{equation}
introducing a scheme specific block-wise transition matrix $\transitionA^{scheme} = \left( \begin{array}{cc}
        \transitionA_{LL}^{scheme} & \transitionA_{LR}^{scheme} \\
        \transitionA_{RL}^{scheme} & \transitionA_{RR}^{scheme} \\
      \end{array} \right)$.

The small relaxation time $\epsilon_L$ in the stiff part of the system leads to a severe time step constraint. In order to design a tailored numerical integration scheme for the decoupled system, we employ a different time integration scheme in each domain. In the non-stiff domain, a standard forward Euler scheme with time step size $\Delta t$ is applied. In the stiff domain, a different scheme is necessary.
%
We derive the following new schemes
\begin{itemize}
  \item[AFE:] Stiff domain: Forward Euler scheme; Non-stiff domain: Forward Euler scheme
  \item[APFE:] Stiff domain: Projective Forward Euler scheme; Non-stiff domain: Forward Euler scheme
  \item[APPFE:] Stiff domain: Projective Forward Euler scheme; Non-stiff domain: Projective Forward Euler scheme
\end{itemize}
We denote the methods as Adaptive Forward Euler scheme (AFE), Adaptive Projective Forward Euler scheme (APFE) and Adaptive Projective Projective Forward Euler scheme (APFE), respectively.

\subsection{Adaptive Forward Euler scheme (AFE)}
\label{sec:AFE}
We first consider a standard forward Euler scheme with a smaller time step $\delta t$ in the stiff region, while using a large time step $\Delta t$ in the non-stiff region. For simplicity, we only consider the case $\Delta t = (K+1) \delta t$ with integer $K \in \mathbb{N}$. The scheme is outlined in figure \ref{fig:AFE_grid}.

\begin{figure}[htb!]
    \centering
    \includegraphics[width=0.75\textwidth]{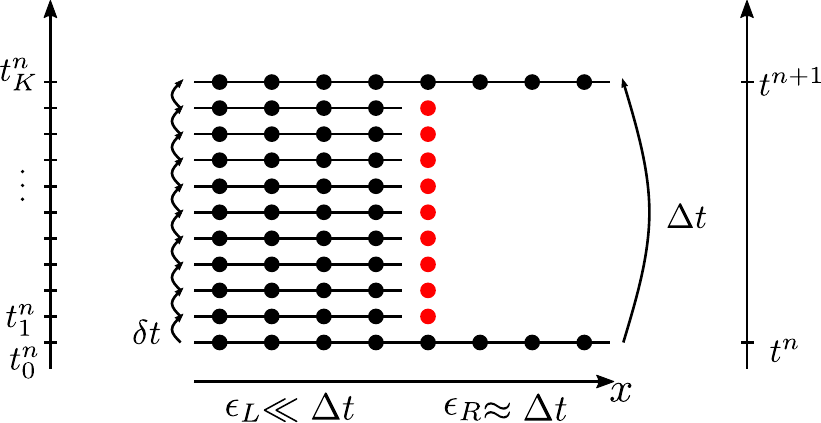}
    \caption{Adaptive forward Euler scheme (AFE) with small time step $\delta t$ in stiff region (left) and large time step $\Delta t$ in non-stiff region (right). Values of red cells at the boundary of the two domains need to be reconstructed.}
    \label{fig:AFE_grid}
\end{figure}

The updates from the values $\Vect{W}^n$ to $\Vect{W}^{n+1}$ are thus performed in the following way
\begin{equation}\label{FE_update_R}
    \Vect{W}_R^{n+1} = \Vect{W}_R^{n} + \Delta t \left( \systemA_{RL} \cdot \Vect{W}_L^{n} + \systemA_{RR} \cdot \Vect{W}_R^{n} \right),
\end{equation}
for the non-stiff part of the domain using a forward Euler step with time step size $\Delta t$, see \Eqn \eqref{e:FE_update}, and
\begin{eqnarray}\label{FE_update_L}
    \Vect{W}_L^{n,k+1} &=& \Vect{W}_L^{n,k} + \delta t \left( \systemA_{LL} \cdot \Vect{W}_L^{n,k} + \systemA_{LR} \cdot \Vect{W}_R^{n,k} \right), k = 0,\ldots,K\\
    \Vect{W}_L^{n+1} &=& \Vect{W}_L^{n,K+1},
\end{eqnarray}
for the stiff part of the domain using a forward Euler step with time step size $\delta t$ and initialisations $\Vect{W}_L^{n,0} = \Vect{W}_L^{n}$, $\Vect{W}_R^{n,0} = \Vect{W}_R^{n}$. The intermediate values $\Vect{W}_R^{n,k}$ needed from the non-stiff part are computed via interpolation, i.e.,
\begin{eqnarray}\label{interpolation_R}
    \Vect{W}_R^{n,k+1} &=& \Vect{W}_R^{n} + (k+1) \cdot \delta t \cdot \frac{\Vect{W}_R^{n+1}-\Vect{W}_R^{n}}{\Delta t}, \\
                       &=& (k+1) \delta t \systemA_{LR} \Vect{W}_L^{n} + (\Vect{I}+ \delta t (k+1) \systemA_{RR} \Vect{W}_R^{n},
\end{eqnarray}
where the sparse form of the off-diagonal parts $\systemA_{LR}$ and $\systemA_{RL}$ allows for an efficient computation of the interpolation only at the interface.

\begin{theorem}\label{th:AFE}
    One time step of the AFE method with time step size $\Delta t$ in the non-stiff domain and time step size $\delta t$ in the stiff domain, for $\Delta t = (K+1) \delta t$ is given by the transition matrix $\transitionA^{AFE}$ with block entries
    \begin{eqnarray*}
        \transitionA^{AFE}_{LL} &=& \left( \Vect{I} + \delta t \systemA_{LL} \right)^{K+1} + \delta t^2 \sum_{k=0}^{K} (K-k) \left( \Vect{I} + \delta t \systemA_{LL} \right)^{k} \systemA_{LR} \systemA_{RL}\\
        \transitionA^{AFE}_{LR} &=& \delta t \sum_{k=0}^{K} \left( \Vect{I} + \delta t \systemA_{LL} \right)^{k} \systemA_{LR} \left( \Vect{I} + (K-k)\delta t \systemA_{RR} \right)\\
        \transitionA^{AFE}_{RL} &=& \Delta t \systemA_{RL} \\
        \transitionA^{AFE}_{RR} &=& \Vect{I} + \Delta t \systemA_{RR}.
    \end{eqnarray*}
\end{theorem}
\begin{proof}
    The resulting blocks of the transition matrix \eqref{decoupled-scheme} are obtained by insertion of the non-stiff entries via \eqref{FE_update_R} and the stiff entries \eqref{FE_update_L} together with the boundary interpolation via \eqref{interpolation_R}.
\end{proof}

As an example, we consider $K=1$, such that $\Delta t = 2\delta t$. Theorem \ref{th:AFE} then leads to the following transition matrix:
\begin{equation}
    \transitionA^{AFE} = \left( \begin{array}{cc}
        \left( \Vect{I} + \delta t \systemA_{LL} \right)^2 + \delta t^2 \systemA_{LR} \systemA_{RL}& 2 \delta t \systemA_{LR} + \delta t^2 \left(\systemA_{LR} \systemA_{RR} + \systemA_{LL} \systemA_{LR}\right)\\
        \Delta t \systemA_{RL} & \Vect{I} + \Delta t \systemA_{RR}
      \end{array} \right),
\end{equation}
which can be written as
\begin{equation}\label{e:AFE_ex}
     \transitionA^{AFE} = \Vect{I} + \Delta t \left( \begin{array}{cc}
        \systemA_{LL} & \systemA_{LR} \\
        \systemA_{RL} & \systemA_{RR} \\
      \end{array} \right) + \frac{\Delta t^2}{4}\left( \begin{array}{cc}
        \systemA_{LL}^2 + \systemA_{LR} \systemA_{RL} & \systemA_{LL} \systemA_{LR} + \systemA_{LR} \systemA_{RL}\\
        \Vect{0} & \Vect{0} \\
      \end{array} \right)
\end{equation}

Comparing \Eqn \eqref{e:AFE_ex} with a Taylor expansion of the exact solution of \Eqn \eqref{e:decoupled-scheme} around $W^n$, i.e., $W(t+\Delta t)= W^n + \Delta t \systemA W^n + \frac{\Delta t^2}{2} \systemA^2 W^n + \mathcal{O}\left(\Delta t^3\right)$, it is clear that the scheme has an error of $\|W^{n+1}- W(t+\Delta t)\|=\mathcal{O}(\Delta t^2)$, such that it is first order accurate in time.

The stability analysis of the scheme is not based on the scalar model equation $\partial_t W = \lambda W$, with $\lambda \in \mathbb{C}^-$, but on the following two-dimensional model
\begin{equation}\label{e:model_eqn}
  \partial_t \left( \begin{array}{c}
        W_L \\
        W_R \\
      \end{array} \right) =
      \left( \begin{array}{cc}
        \lambda_L & 0 \\
        0 & \lambda_R \\
      \end{array} \right)
       \left( \begin{array}{c}
        W_L \\
        W_R \\
      \end{array} \right),
\end{equation}
for two variables $W_L,W_R$ following two scales $\lambda_L,\lambda_R \in \mathbb{C}^-$, reflecting spatially varying relaxation times.

This leads to the following transition matrix for the model equation \eqref{e:model_eqn}
\begin{equation}\label{e:AFE_stab_matrix}
    \transitionA^{AFE} = \left( \begin{array}{cc}
        \left(1+\delta t \lambda_L\right)^{K+1} & 0 \\
        0 & 1+\Delta t \lambda_R \\
      \end{array} \right)
\end{equation}

The stability domain of the AFE scheme, derived using the transition matrix from \Eqn \eqref{e:AFE_stab_matrix} $\|\transitionA^{AFE}\| \leq 1$ is then given by
\begin{equation}\label{e:AFE_stab_region}
    \lambda_L \in C\left( - \frac{1}{\delta t}, \frac{1}{\delta t} \right) \textrm{ and } \lambda_R \in C\left( - \frac{1}{\Delta t}, \frac{1}{\Delta t} \right),
\end{equation}
and shown in \Fig \ref{fig:AFE_stab_region}.

\begin{figure}[htb!]
        \centering
        \begin{subfigures}
        \subfloat[Stability region of Adaptive Forward Euler scheme (AFE). Stiff domain uses small time step $\delta t$ for all modes and non-stiff domain uses larger $\Delta t$ for all modes. \label{fig:AFE_stab_region}
        ]{\includegraphics[width=0.7\linewidth]{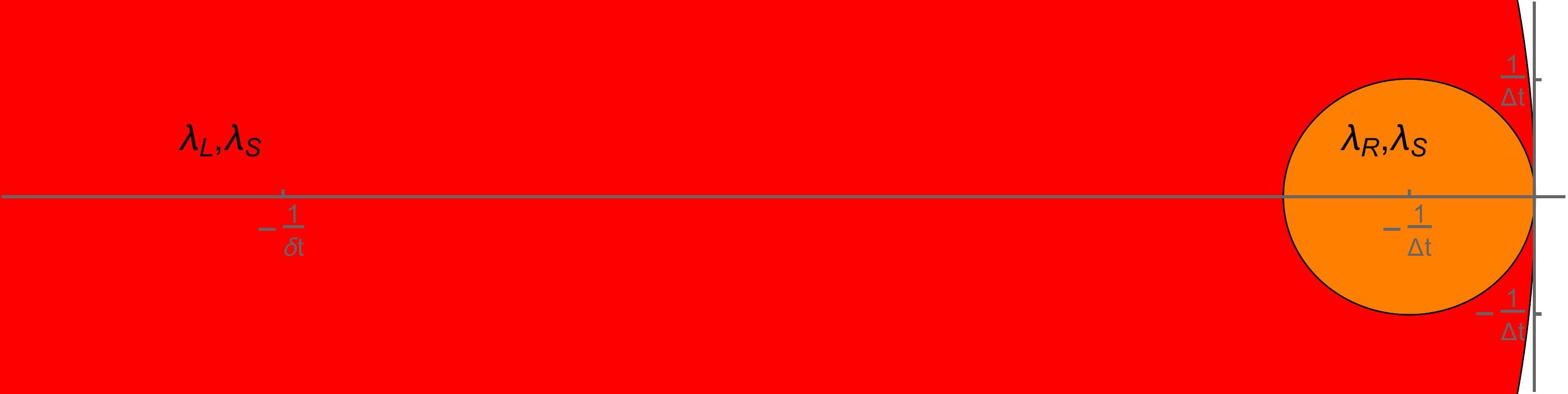}}\\
        \subfloat[Stability region of Adaptive Projective Forward Euler scheme (APFE). Stiff domain uses one inner time step $\delta t$ for fast modes and one time step $\Delta t$ for the other modes. Non-stiff domain uses single time step $\Delta t$ for all modes. \label{fig:APFE_stab_region}
        ]{\includegraphics[width=0.7\linewidth]{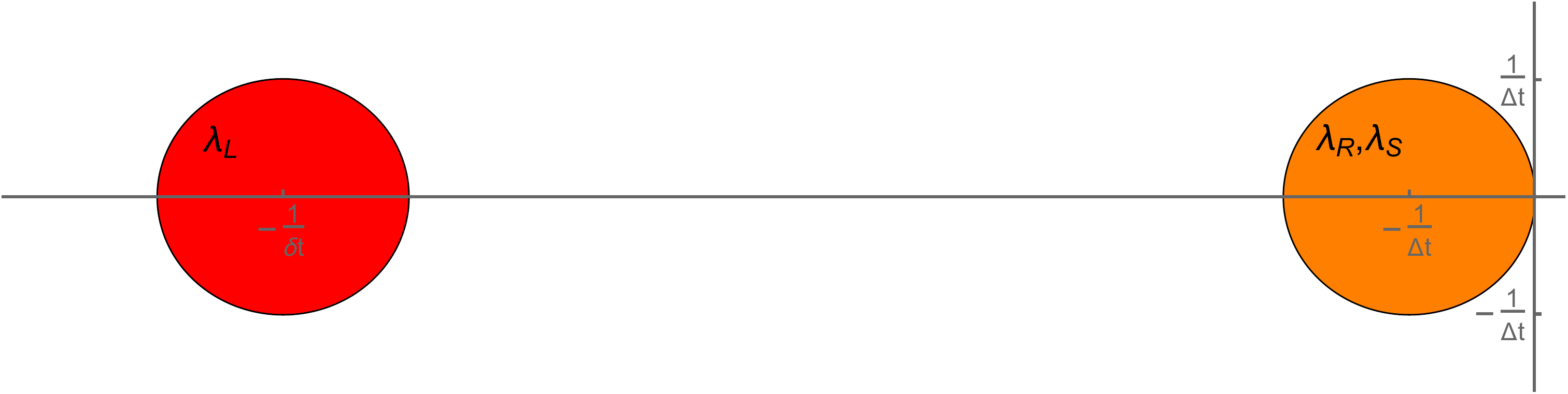}}\\
        \subfloat[Stability region of Adaptive Projective Projective Forward Euler scheme (APPFE). Both the stiff domain and the non-stiff domain each use one inner time step $\delta t_{L,R}$ for fast modes and one time step $\Delta t$ for the other modes. \label{fig:APPFE_stab_region}
        ]{\includegraphics[width=0.7\linewidth]{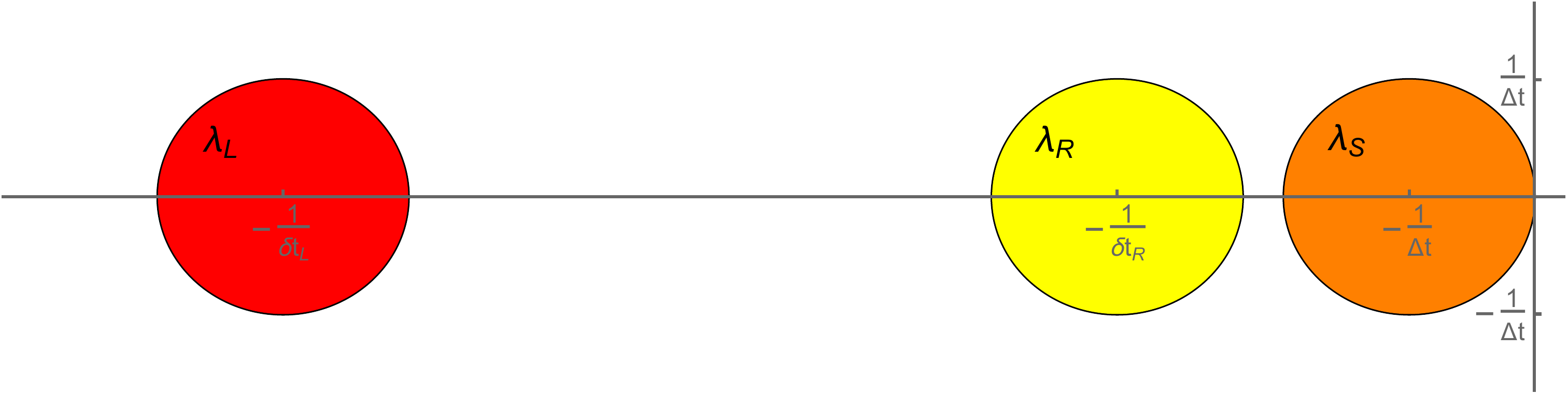}}
        \end{subfigures}
        \caption{Stability regions of adaptive time integration schemes. The respective eigenvalues need to be located within the specified domains for stability of the scheme. 
        Orange denotes stability region for the whole domain. Red denotes separate stability region for the stiff domain. Yellow denotes separate stability region of non-stiff domain.}
        \label{fig:stab_regions_adaptive}
\end{figure}

Using the spectral analysis of the previous section, we can derive the respective bounds on the parameters $\delta t$ and $\Delta t$ depending on the spatial discretization, the $CFL$ number, and the relaxation times $\epsilon_{L,R}$.
In order to include the whole spectrum for the model analyzed in Theorem \ref{th:spectrum} within the stability domain of the AFE scheme, we now consider a spatially varying relaxation time with discrete values $\epsilon_L \ll \epsilon_R$. We then determine the parameters based on \Eqn \eqref{e:AFE_stab_region} and Theorem \ref{th:spectrum} as
\begin{eqnarray}
  \frac{1}{\delta t} &\geq& \frac{1}{2}\left(-\lambda_{\epsilon_L} + R \right) \label{e:AFE_stab_cond1}\\
  \frac{1}{\delta t} &\geq& R \label{e:AFE_stab_cond2} \\
  \frac{1}{\Delta t} &\geq& \frac{1}{2}\left(-\lambda_{\epsilon_R} + R \right) \label{e:AFE_stab_cond3}\\
  \frac{1}{\Delta t} &\geq& R \label{e:AFE_stab_cond4}
\end{eqnarray}

Inserting $\Delta t = CFL \frac{\Delta x}{\lambda_{max}}$ and the known values of $\lambda_{\epsilon_L}, \lambda_{\epsilon_R}$ and $R$ for the different schemes yields:
\begin{itemize}
  \item[1.] the upwind scheme is conditionally stable for $\delta t = \frac{1}{\frac{\lambda_{max}}{\Delta x} + \frac{1}{2\epsilon_L}} = \mathcal{O}(\epsilon_L)$, and  $CFL \leq \frac{1}{\frac{\Delta x}{2\epsilon_R \lambda_{max}}+1}$.
  \item[2.] the Lax-Friedrichs scheme is unconditionally unstable because the intermediate cluster cannot be integrated in a stable way.
  \item[3.] the FORCE scheme is conditionally stable for $\delta t = \frac{1}{\frac{\lambda_{max}}{2\Delta x}\left(\frac{1}{CFL} + CFL\right)  + \frac{1}{2\epsilon_L}} = \mathcal{O}(\epsilon_L)$, and  $CFL \leq - \frac{\Delta x}{2\epsilon_R \lambda_{max}} + \sqrt{\frac{\Delta x}{2\epsilon_R \lambda_{max}}^2+1}$.
\end{itemize}

Note that in comparison to the FE scheme \ref{sec:FE}, only the small time step size $\delta t$ is used to resolve the stiff domain corresponding to $\epsilon_L$, whereas the rest of the domain can use a standard time step $\Delta t$ given by $CFL = \mathcal{O}(1)$ for a larger $\epsilon_R$. However, the Lax-Friedrichs scheme is still unstable and a $CFL$ condition remains for the other schemes. In addition, many steps with $\delta t$ need to be performed in the stiff region.

The eigenvalues of the transition matrix $\transitionA_{AFE}$ with an upwind spatial discretization and parameters according to the aforementioned stability conditions are plotted in \Fig \ref{fig:AFE_transition_stab_region}. Again, all eigenvalues are inside the unit circle and we conclude that the method is indeed stable for the parameter settings predicted by our analysis. The eigenvalues $\lambda_i$ are very close to the stability boundary $\|\lambda_i \| < 1$, which indicates that both the estimates of the spectrum of the model equation and the stability properties of the scheme are relatively sharp.

\begin{figure}[htb!]
        \centering
        \begin{subfigures}
        \subfloat[Adaptive Forward Euler. \label{fig:AFE_transition_stab_region}
        ]{\includegraphics[width=0.46\linewidth]{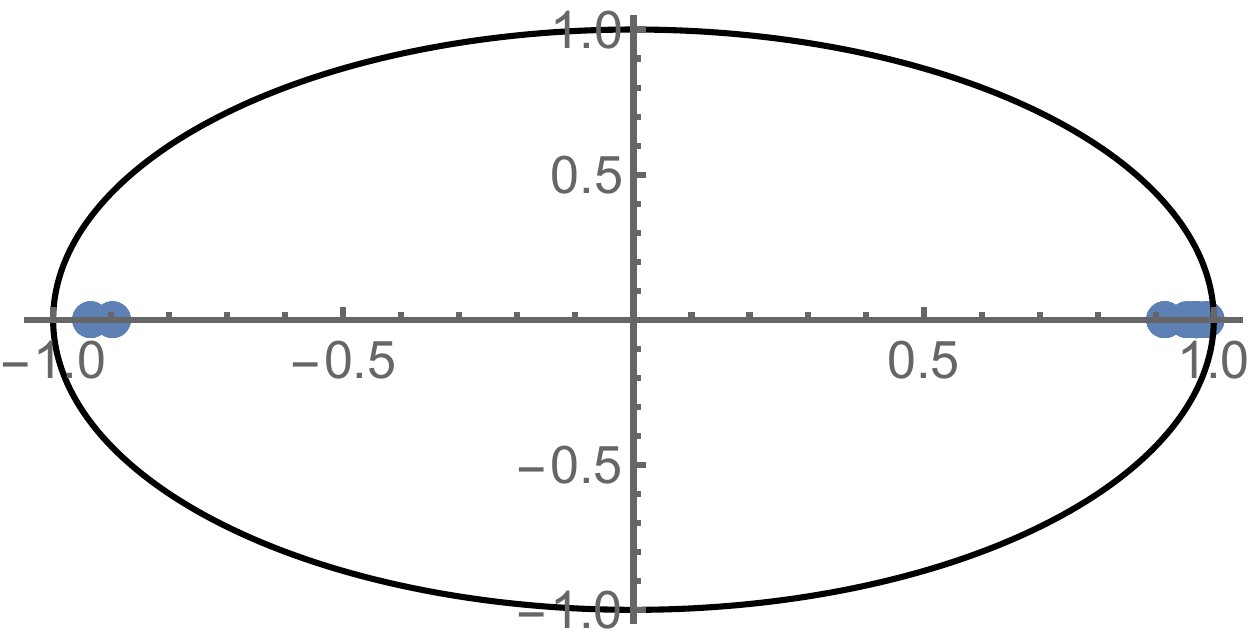}}
        \subfloat[Adaptive Projective Forward Euler. \label{fig:APFE_transition_stab_region}
        ]{\includegraphics[width=0.46\linewidth]{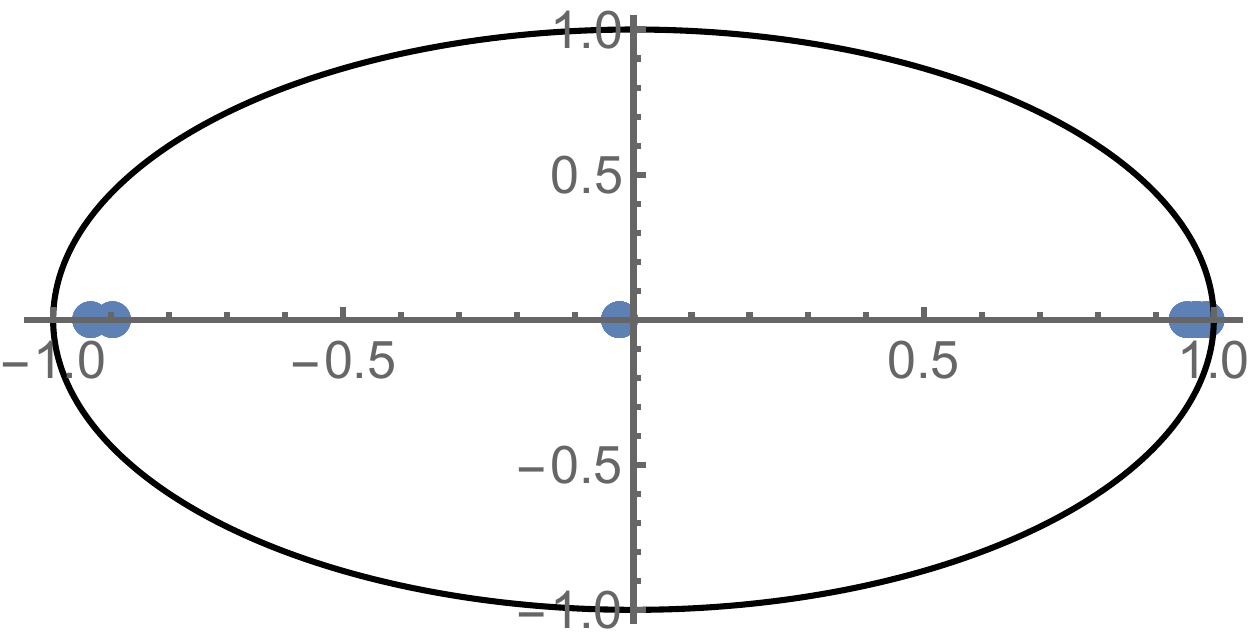}}
        \end{subfigures}
        \caption{Numerical spectrum of the transition matrix $\transitionA$ for Adaptive Forward Euler (left) and Adaptive Projective Forward Euler (right). Both schemes are stable if parameters are chosen according to the derived analytical values, while the estimates are relatively sharp as eigenvalues are close to stability boundary. Upwind spatial discretization, $(\rho,u,\theta)=(1,\pi,1)$, i.e. $\lambda_{max}\approx 6$, $\epsilon_L=10^{-4},\epsilon_L=10^{-3}$, $\Delta x = 1/10$. }
        \label{fig:adaptive_transition_stab_regions}
\end{figure}

\subsection{Adaptive Projective Forward Euler scheme (APFE)}
\label{sec:APFE}
We keep a standard forward Euler scheme with large time step $\Delta t$ in the non-stiff region but employ a Projective Forward Euler scheme with $K$ inner Forward Euler steps of smaller time step $\delta t \ll \Delta t$ in the stiff domain. The idea is outlined in \Fig \ref{fig:APFE_grid}.
\begin{figure}[htb!]
    \centering
    \includegraphics[width=0.75\textwidth]{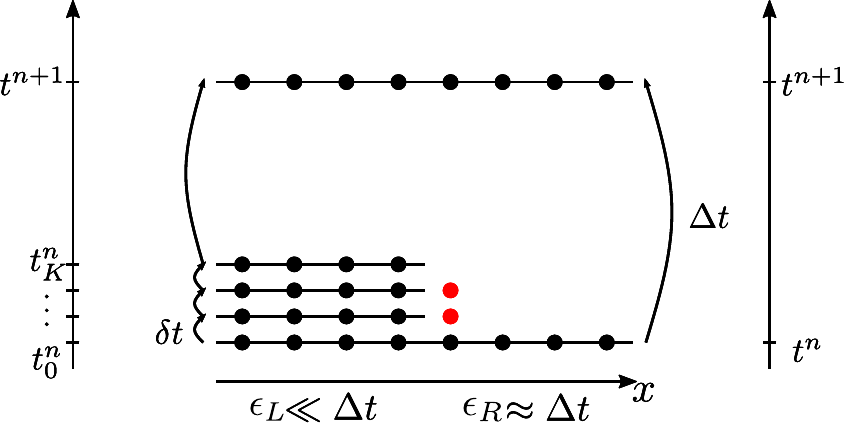}
    \caption{Adaptive projective forward Euler scheme (APFE) with $K$ inner small time steps $\delta t$ in stiff region (left) and large time step $\Delta t$ in non-stiff region (right). Values of red cells at the boundary of the two domains need to be reconstructed.}
    \label{fig:APFE_grid}
\end{figure}

The update and interpolation of the non-stiff values from \Eqns \eqref{FE_update_R} and \eqref{interpolation_R} are the same. The Projective Forward Euler scheme uses $K$ inner Forward Euler steps and subsequent extrapolation over the remaining time interval
\begin{eqnarray}\label{AFE_update_L}
    \Vect{W}_L^{n,k+1} &=& \Vect{W}_L^{n,k} + \delta t \left( \systemA_{LL} \cdot \Vect{W}_L^{n,k} + \systemA_{LR} \cdot \Vect{W}_R^{n,k} \right), k = 0,\ldots,K\\
    \Vect{W}_L^{n+1} &=& \Vect{W}_L^{n,K+1} + \left(\Delta t - (K+1) \delta t\right) \frac{\Vect{W}_L^{n,K+1} - \Vect{W}_L^{n,K}}{\delta t},
\end{eqnarray}

\begin{theorem}\label{th:APFE}
    One time step of the APFE method with time step size $\Delta t$ in the non-stiff domain and $K$ time steps of size $\delta t$ with subsequent extrapolation in the stiff domain is given by the transition matrix $\transitionA^{APFE}$ with block entries
    \begin{eqnarray*}
        \transitionA^{APFE}_{LL} &=& \left( \Vect{I} + (\Delta t - K \delta t) \systemA_{LL} \right) \left( \delta t^2 \sum_{k=0}^{K-1} (K-1-k) \left( \Vect{I} + \delta t \systemA_{LL} \right)^k \systemA_{LR} \systemA_{RL} + \left( \Vect{I} + \delta t \systemA_{LL} \right)^K \right) \\
                      & & + (\Delta t -K \delta t) \systemA_{LR} K \delta t \systemA_{RL}  \\
        \transitionA^{APFE}_{LR} &=& \left( \Vect{I} + (\Delta t - K \delta t) \systemA_{LL} \right) \delta t \sum_{k=0}^{K-1} \left( \Vect{I} + \delta t \systemA_{LL} \right)^k \systemA_{LR} \left( \Vect{I} + (K-1-k)\delta t \systemA_{RR} \right) \\
                      & & + (\Delta t -K \delta t) \systemA_{LR} \left( \Vect{I} + K \delta t \systemA_{RR} \right) \\
        \transitionA^{APFE}_{RL} &=& \Delta t \systemA_{RL} \\
        \transitionA^{APFE}_{RR} &=& \Vect{I} + \Delta t \systemA_{RR}.
    \end{eqnarray*}
\end{theorem}
\begin{proof}
    The resulting blocks of the transition matrix \eqref{decoupled-scheme} are obtained by insertion of the non-stiff entries via \eqref{FE_update_R} and the stiff entries \eqref{AFE_update_L} together with the boundary interpolation via \eqref{interpolation_R}.
\end{proof}

As an example, we consider $K=1$, which is often used for PFE schemes. Theorem \ref{th:APFE} then leads to the following transition matrix: $\transitionA^{APFE} = $
\begin{equation}
    \left( \begin{array}{cc}
        \left( \Vect{I} + (\Delta t - \delta t) \systemA_{LL} \right) \left( \Vect{I} + \delta t \systemA_{LL} \right) + \delta t (\Delta t - \delta t) \systemA_{LR} \systemA_{RL} & (\Delta t - \delta t) \left( \systemA_{LR} + \delta t \systemA_{LR} \systemA_{RR} \right) \\
        \Delta t \systemA_{RL} & \Vect{I} + \Delta t \systemA_{RR}
      \end{array} \right),
\end{equation}
which can be written as $\transitionA^{APFE} = $
\begin{equation}\label{e:APFE_ex}
    \Vect{I} + \Delta t \left( \begin{array}{cc}
        \systemA_{LL} & \systemA_{LR} \\
        \systemA_{RL} & \systemA_{RR} \\
      \end{array} \right) + \left( \begin{array}{cc}
        \delta t (\Delta t - \delta t)  \left( \systemA_{LL}^2 + \systemA_{LR} \systemA_{RL} \right) & -\delta t \systemA_{LR} + \delta t (\Delta t - \delta t) \systemA_{LR} \systemA_{RR}\\
        \Vect{0} & \Vect{0} \\
      \end{array} \right)
\end{equation}

Considering consistency, we can again compare \Eqn \eqref{e:APFE_ex} with a Taylor expansion of the exact solution of \Eqn \eqref{e:decoupled-scheme} and obtain that the scheme has an error of $\|W^{n+1}- W(t+\Delta t)\|=\mathcal{O}(\Delta t^2)$, i.e., it is first order accurate in time.

The stability analysis is again based on \Eqn \eqref{e:model_eqn}, which leads to the following transition matrix
\begin{equation}\label{e:APFE_stab_matrix}
    \transitionA^{APFE} = \left( \begin{array}{cc}
        \left(1+\left(\frac{\Delta t}{\delta t}- K \right) \delta t \lambda_L \right)\left( 1 + \delta t \lambda_L \right)^K & 0 \\
        0 & 1+\Delta t \lambda_R \\
      \end{array} \right)
\end{equation}

The stability domain of the APFE scheme is derived in the same fashion as for the AFE scheme using $\|\transitionA^{APFE}\| \leq 1$ and given by
\begin{equation}\label{e:APFE_stab_region}
    \lambda_L \in C\left( - \frac{1}{\Delta t}, \frac{1}{\Delta t} \right) \cup C\left( - \frac{1}{\delta t}, \frac{1}{\delta t}\left(\frac{\delta t}{\Delta t}\right)^{K+1}\right) \textrm{ and } \lambda_R \in C\left( - \frac{1}{\Delta t}, \frac{1}{\Delta t} \right),
\end{equation}
as shown in \Fig \ref{fig:APFE_stab_region}.

Using the spectral analysis of the previous section, we can derive the respective bounds on the parameters $\delta t$ and $\Delta t$, and $K$ depending on the spatial discretization, the $CFL$ number, and the relaxation times $\epsilon_{L,R}$.
In order to include the whole spectrum for the model analyzed in Theorem \ref{th:spectrum} within the stability domain of the APFE scheme, we consider the same spatially varying relaxation time with discrete values $\epsilon_L \ll \epsilon_R$. We then determine the parameters based on \Eqn \eqref{e:APFE_stab_region} and \ref{th:spectrum} as
\begin{eqnarray}
  \frac{1}{\delta t} &=& -\lambda_{\epsilon_L} \label{e:APFE_stab_cond1}\\
  \frac{1}{\delta t}\left(\frac{\delta t}{\Delta t}\right)^{K+1} &\geq& R \label{e:APFE_stab_cond2} \\
  \frac{1}{\Delta t} &\geq& \frac{1}{2}\left(-\lambda_{\epsilon_R} + R \right) \label{e:APFE_stab_cond3}\\
  \frac{1}{\Delta t} &\geq& R \label{e:APFE_stab_cond4}
\end{eqnarray}

Inserting $\Delta t = CFL \frac{\Delta x}{\lambda_{max}}$ and the known values of $\lambda_{\epsilon_L}, \lambda_{\epsilon_R}$ and $R$ for the different schemes yields:
\begin{itemize}
  \item[1.] the upwind scheme is conditionally stable for $\delta t = \frac{1}{\frac{\lambda_{max}}{\Delta x} + \frac{1}{\epsilon_L}} = \mathcal{O}(\epsilon_L)$, $K=1$, and  $CFL \leq \frac{1}{\frac{\Delta x}{2\epsilon_R \lambda_{max}}+1}$.
  \item[2.] the Lax-Friedrichs scheme is unconditionally unstable because the intermediate cluster cannot be integrated in a stable way.
  \item[3.] the FORCE scheme is conditionally stable for $\delta t = \frac{1}{\frac{\lambda_{max}}{2\Delta x}\left(\frac{1}{CFL} + CFL\right)  + \frac{1}{\epsilon}} = \mathcal{O}(\epsilon_L)$, $K=1$, and  $CFL \leq - \frac{\Delta x}{2\epsilon_R \lambda_{max}} + \sqrt{\frac{\Delta x}{2\epsilon_R \lambda_{max}}^2+1}$.
\end{itemize}

Note that the value $K=1$ is chosen here for convenience. Other values are possible and extend the stability region towards the slow cluster, see \cite{Melis2019}.

The eigenvalues of the transition matrix $\transitionA_{APFE}$ with an upwind spatial discretization and parameters according to the aforementioned stability conditions are plotted in \Fig \ref{fig:APFE_transition_stab_region}. Again, all eigenvalues are inside the unit circle and we conclude that the method is indeed stable for the parameter settings predicted by our analysis. The eigenvalues $\lambda_i$ are very close to the stability boundary $\|\lambda_i \| < 1$, which indicates that both the estimates of the spectrum of the model equation and the stability properties of the scheme are relatively sharp.

In comparison to the AFE scheme above, the APFE scheme uses less small time steps $\delta t$ in the stiff region, while performing the same large time step $\Delta t$ in the non-stiff region. The speedup is thus purely due to a more efficient integration of the stiff terms in the stiff region. Due to the relaxation time in the non-stiff region, the Lax-Friedrichs scheme is still unstable, as in the case of the AFE method.

\subsection{Adaptive Projective Projective Forward Euler (APPFE)}
\label{sec:APPFE}
The APFE method is already able to overcome the stability constraints in the stiff region with relaxation time $\epsilon_L$. However, in order to overcome a potential stability constraint in the other part of the domain with $\epsilon_R$, a standard FE method in that domain is not enough. We will therefore introduce an APPFE method, that uses a PFE method in both regions, but adapts the inner time step size $\delta t$ to the respective relaxation times. The idea is outlined in \Fig \ref{fig:APPFE3_grid}.
\begin{figure}[htb!]
    \centering
    \includegraphics[width=0.75\textwidth]{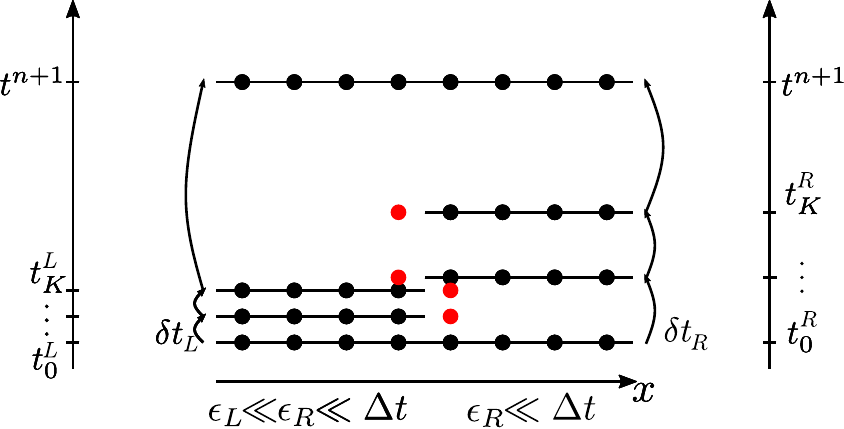}
    \caption{Adaptive projective projective forward Euler scheme (APPFE) with $K+1=3$ inner small time steps $\delta t_L$ in stiff region (left) and inner small time steps $\delta t_R > \delta t_L$ in semi stiff region (right). Values of red cells at both sides of the boundary of the two domains need to be reconstructed. }
    \label{fig:APPFE3_grid}
\end{figure}

For simplicity, we assume that there is a significant gap between the relaxation times $\epsilon_L$ and $\epsilon_R$, resulting in $\delta_L \ll \delta_R$, such that the small time steps and the extrapolation are not intertwined, as shown in \Fig \ref{fig:APPFE3_grid}. The update and interpolation of boundary values are then  same as in the previous schemes. The Projective Forward Euler schemes uses $K_L$ and $K_R$ inner Forward Euler steps, respectively, and perform a subsequent extrapolation over the remaining respective time interval, i.e.,
\begin{eqnarray}
    \Vect{W}_L^{n,k+1} &=& \Vect{W}_L^{n,k} + \delta t_L \left( \systemA_{LL} \cdot \Vect{W}_L^{n,k} + \systemA_{LR} \cdot \Vect{W}_{RL}^{n,k} \right), k = 0,\ldots,K\\
    \Vect{W}_L^{n+1} &=& \Vect{W}_L^{n,K_L+1} + \left(\Delta t_L - (K_L+1) \delta_L t\right) \frac{\Vect{W}_L^{n,K_L+1} - \Vect{W}_L^{n,K_L}}{\delta t_L}, \\
    \Vect{W}_R^{n,k+1} &=& \Vect{W}_R^{n,k} + \delta t_R \left( \systemA_{RL} \cdot \Vect{W}_LR^{n,k} + \systemA_{RR} \cdot \Vect{W}_{R}^{n,k} \right), k = 0,\ldots,K\\
    \Vect{W}_R^{n+1} &=& \Vect{W}_R^{n,K_R+1} + \left(\Delta t_R - (K_R+1) \delta_R t\right) \frac{\Vect{W}_R^{n,K_R+1} - \Vect{W}_R^{n,K_R}}{\delta t_R},
\end{eqnarray}
where the necessary boundary values $\Vect{W}_{RL}^{n,k}$ for the left update and $\Vect{W}_LR^{n,k}$ for the right update are obtained via interpolation, i.e.,
\begin{eqnarray}
    \Vect{W}_{RL}^{n,k} &=& \Vect{W}_R^{n} + (k+1) \cdot \delta t_L \cdot \left( \systemA_{RL} \cdot \Vect{W}_L^{n} + \systemA_{RR} \cdot \Vect{W}_{R}^{n} \right), \label{interpolation_APPFE_R}\\
    \Vect{W}_{LR}^{n,k} &=& \Vect{W}_L^{n} + \left( (k+1)\delta t_R - K_L\delta t_L \right)\frac{\Vect{W}_L^{n,K_L+1}- \Vect{W}_L^{n,K_L}}{\delta t_L}. \label{interpolation_APPFE_L}
\end{eqnarray}

As the transition matrix is a lengthy expression that has no further use for us expect for the stability analysis, we omit its rather tedious derivation here and focus on the stability properties, which can be obtained from the definition of the scheme applied to the model system \Eqn \eqref{e:model_eqn}. In this case, the transition matrix reads
\begin{equation}\label{e:APPFE_stab_matrix}
    \transitionA^{APPFE} = \left( \begin{array}{cc}
        \left(1+\left(\frac{\Delta t}{\delta t_L}- K_L \right) \delta t_L \lambda_L \right)\left( 1 + \delta t_L \lambda_L \right)^{K_L} & 0 \\
        0 & \left(1+\left(\frac{\Delta t}{\delta t_R}- K_R \right) \delta t_R \lambda_R \right)\left( 1 + \delta t_R \lambda_R \right)^{K_R} \\
      \end{array} \right)
\end{equation}

The stability domain of the APPFE scheme is then derived using $\|\transitionA^{APPFE}\| \leq 1$ and is given by
\begin{equation}\label{e:APPFE_stab_region1}
    \lambda_L \in C\left( - \frac{1}{\delta t_L}, \frac{1}{\Delta t} \right) \cup C\left( - \frac{1}{\delta t_L}, \frac{1}{\delta t_L}\left(\frac{\delta t_L}{\Delta t}\right)^{K_L+1}\right)
\end{equation}
and
\begin{equation}\label{e:APPFE_stab_region2}
    \lambda_R \in C\left( - \frac{1}{\delta t_R}, \frac{1}{\Delta t} \right) \cup C\left( - \frac{1}{\delta t_R}, \frac{1}{\delta t_R}\left(\frac{\delta t_R}{\Delta t}\right)^{K_R+1}\right),
\end{equation}
as shown in \Fig \ref{fig:APPFE_stab_region}.

Using the spectral analysis of the previous section, we can derive the respective bounds on the parameters $\delta t_L, K_L, \delta t_R, K_R$, and $\Delta t$ depending on the spatial discretization, the $CFL$ number, and the relaxation times $\epsilon_{L,R}$.
In order to include the whole spectrum for the model analyzed in Theorem \ref{th:spectrum} within the stability domain of the APFE scheme, we consider the same spatially varying relaxation time with discrete values $\epsilon_L \ll \epsilon_R$. We then determine the parameters based on \Eqns \eqref{e:APPFE_stab_region1} and \eqref{e:APPFE_stab_region2} and Theorem \ref{th:spectrum} as
\begin{eqnarray}
  \frac{1}{\delta t_L} &=& -\lambda_{\epsilon_L} \label{e:APPFE_stab_cond1}\\
  \frac{1}{\delta t_R} &=& -\lambda_{\epsilon_R} \label{e:APPFE_stab_cond2}\\
  \frac{1}{\delta t_L}\left(\frac{\delta t_L}{\Delta t}\right)^{K_L+1} &\geq& R \label{e:APPFE_stab_cond3} \\
  \frac{1}{\delta t_R}\left(\frac{\delta t_R}{\Delta t}\right)^{K_R+1} &\geq& R \label{e:APPFE_stab_cond4}
\end{eqnarray}

Using $\Delta t = CFL \frac{\lambda_{max}}{\Delta x}$ and the known values of $\lambda_{\epsilon_L}, \lambda_{\epsilon_R}$ and $R$ for the different schemes yields:
\begin{itemize}
  \item[1.] the upwind scheme is conditionally stable for $\delta t_{L/R} = \frac{1}{\frac{\lambda_{max}}{\Delta x} + \frac{1}{\epsilon_{L/R}}} = \mathcal{O}(\epsilon_{L/R})$, $K=1$, and  $CFL \leq 1$.
  \item[1.] the Lax-Friedrichs scheme is conditionally stable for $\delta t_{L/R} = \frac{1}{\frac{\lambda_{max}}{CFL \Delta x} + \frac{1}{\epsilon_{L/R}}} = \mathcal{O}(\epsilon_{L/R})$, $K=1$, and  $CFL \leq 1$.
  \item[3.] the FORCE scheme is conditionally stable for $\delta t_{L/R} = \frac{1}{\frac{\lambda_{max}}{2\Delta x}\left(\frac{1}{CFL} + CFL\right)  + \frac{1}{\epsilon_{L/R}}} = \mathcal{O}(\epsilon_{L/R})$, $K=1$, and  $CFL \leq 1$.
\end{itemize}

Again, $K=1$ is chosen here for convenience. Other values are possible and extend the stability region towards the slow cluster, see \cite{Melis2019}.

All the AFE and APFE scheme, the $CFL$ condition is much less restricted and a full convective time step $\Delta t = \frac{\Delta x}{\lambda_{max}}$ with $CFL = 1$ is possible. This reduces the runtime significantly in case of stiff relaxation times.

\begin{remark}
    While focussing the analysis on first-order outer time integrators like the Forward Euler scheme (FE) in this paper, the same analysis and implementation can be performed for higher-order Runge-Kutta schemes, that replace the outer integrator \cite{Lafitte2016,Lafitte2017}. This leads to Adaptive Projective Runge-Kutta schemes (APRK). One numerical example application of a second order APRK based on the Heun method as outer integrator is given in the next section. Another extension is possible for connected spectra via Telescopic Projective Integration schemes (TPI), developed in \cite{Melis2019,Melis2016}.
\end{remark}     
\section{Numerical results}
\label{sec:results}
In this section, we briefly validate the numerical accuracy of the newly derived adaptive projective methods with the help of a two-beam test case and give theoretical results for the potential speedup of our new methods. As the focus of this paper is the derivation and analysis of the new schemes, we do not perform exhaustive tests and simulations of all possible combinations of schemes, spatial discretizations and parameter settings, but leave this for future work.

\subsection{Two-beam test}
The two beam test case is a standard test case for rarefied gases and was used in \cite{Koellermeier2017d}, \cite{Schaerer2015} for different moment models for constant relaxation time $\epsilon$. A spatially varying relaxation time was first tested in \cite{Koellermeier2021}. For more detailed information on the test setup, we refer to the literature.

The initial Riemann data for the left-hand side and the right-hand side of the domain, respectively, is given by
\begin{equation}
    \vect{w}_L = \left( 1,0.5,1,0,\ldots,0\right)^T, \quad \quad \vect{w}_R = \left( 1,-0.5,1,0,\ldots,0\right)^T,
    \label{e:2beam_IC}
\end{equation}
modeling two colliding Maxwellian distributed particle beams. This test case is especially challenging as it is difficult to represent the analytical solution using a polynomial expansion. In the free streaming case $\epsilon = \infty$ the analytical solution is a sum of two Maxwellians according to \cite{Schaerer2015}.

The numerical tests are performed on the computational domain $[-10,10]$, discretized using $500$ points and the end time is $t_{\textrm{END}}=0.1$ using a constant macroscopic time step $\Delta t$ according to a CFL number of $0.5$ for all tests. This results in the macroscopic time step size $\Delta t = 3.85 \cdot 10^{-4}$ for the HME model \eqref{sec:HME} with $M+1=10$ equations, which is used here as one example. Note that extensive tests of the QBME moment model have been performed in \cite{Koellermeier2017d,Schaerer2015} for the rarefied regime and in \cite{Koellermeier2021} in the case of small relaxation time $\epsilon \ll 1$. In the latter case, we can assume that the model error of the moment model can be neglected and do not show a comparison with reference models. For more details on the accuracy of moment models for the two-beam model, we refer to \cite{Koellermeier2017d}.

The spatially varying relaxation time is chosen as
\begin{equation}\label{e:collision_rate_adaptive_2beam}
    \epsilon(x) = \left\{
  \begin{array}{cl}
    \epsilon_L = 10^{-4} & \textrm{if } x < 0, \\
    \epsilon_R = 10^{-2}& \textrm{if } x \geq 0, \\
  \end{array} \right.
\end{equation}

As the spatial discretization method, we use the first order FORCE scheme and compare two methods for the time integration:
\begin{itemize}
  \item[1.] A standard PI scheme using $\delta t = \epsilon_L$, $K=2$, and $\Delta t$ according to a macroscopic $CFL=0.5$.
  \item[2.] An APFE scheme using $\delta t = \epsilon_L$, $K=2$ in the stiff left part of the domain and $\Delta t$ according to a macroscopic $CFL=0.5$ in the right half of the domain.
\end{itemize}
Note that an APPFE method is not necessary here as there is no additional constraint on the time step size in the non-stiff domain due to the relatively fine spatial discretization. When using higher-order spatial discretization and larger time step sizes $\Delta t$, a coarser grid would lead to possible gains for an APPFE (or a higher-order APPRK) method. This is left for future work.

The numerical results shown in \Fig \ref{fig:2beam} clearly show that the adaptive scheme is able to achieve high accuracy in this numerical test. \Fig \ref{fig:API_2beam500V0p5Piecewise100QBME10Kn0p0001APFEFORCE1p} shows that the error with respect to the first order PFE scheme is negligible for the pressure $p$, while \ref{fig:API_2beam500V0p5Piecewise100QBME10Kn0p0001APFEFORCE1Q} shows even less diffusivity for the heat flux $Q$. This is due to the fact that the APFE method performs less time steps in the non-stiff domain, thus decreasing the added numerical diffusion. Comparing the standard PFE scheme with the APPFE scheme, we clearly see that the adaptivity does not induce any oscillations for this test case. For future work, higher-order spatial discretizations and adaptive higher-order time stepping methods like APPRK need to be investigated.
\begin{figure}[htb!]
    \centering
    \begin{subfigures}
    \subfloat[Pressure $p$. \label{fig:API_2beam500V0p5Piecewise100QBME10Kn0p0001APFEFORCE1p}
    ]{\includegraphics[width=0.5\linewidth]{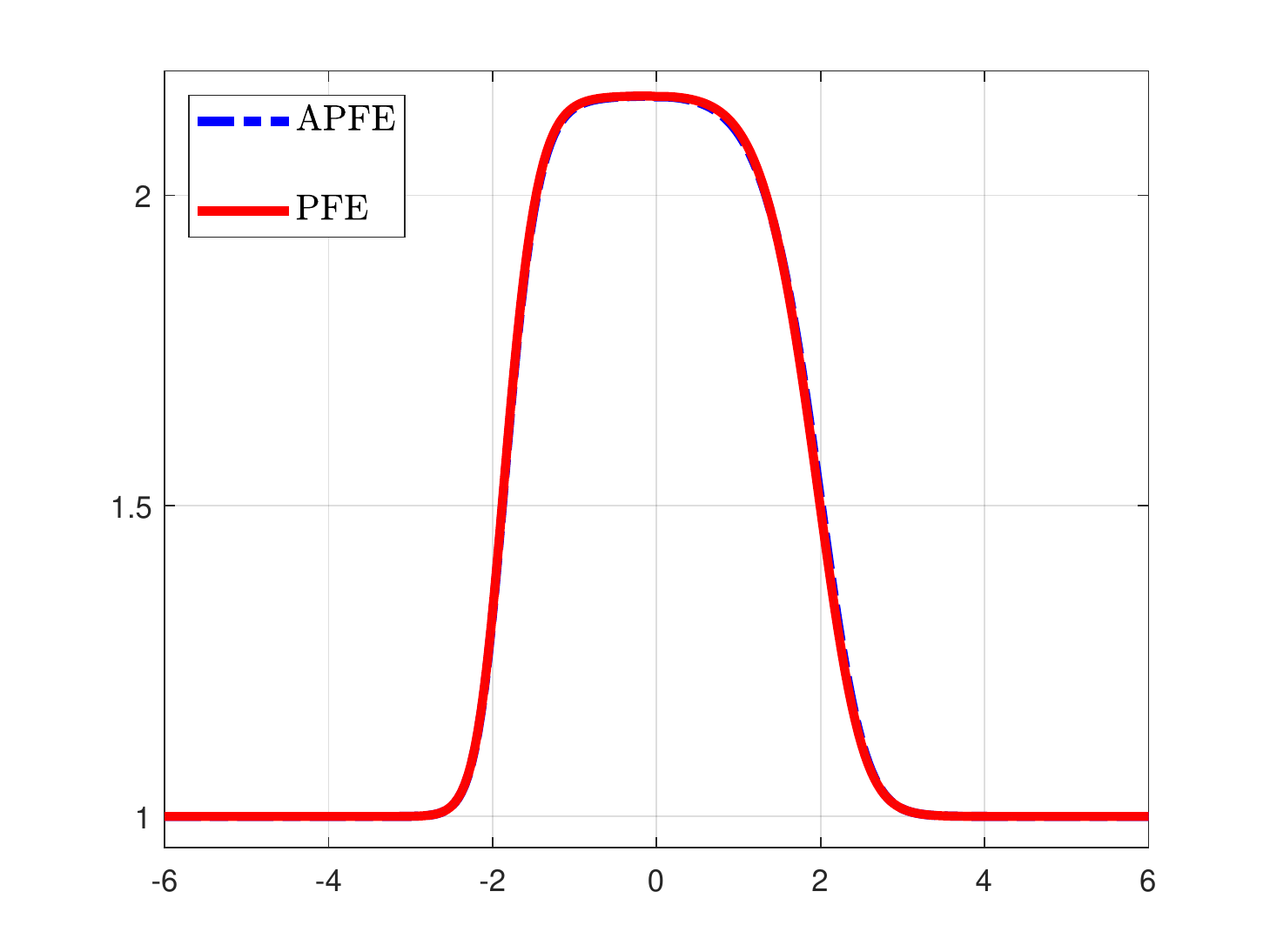}}
    \subfloat[Heat flux $Q$. \label{fig:API_2beam500V0p5Piecewise100QBME10Kn0p0001APFEFORCE1Q}
    ]{\includegraphics[width=0.5\linewidth]{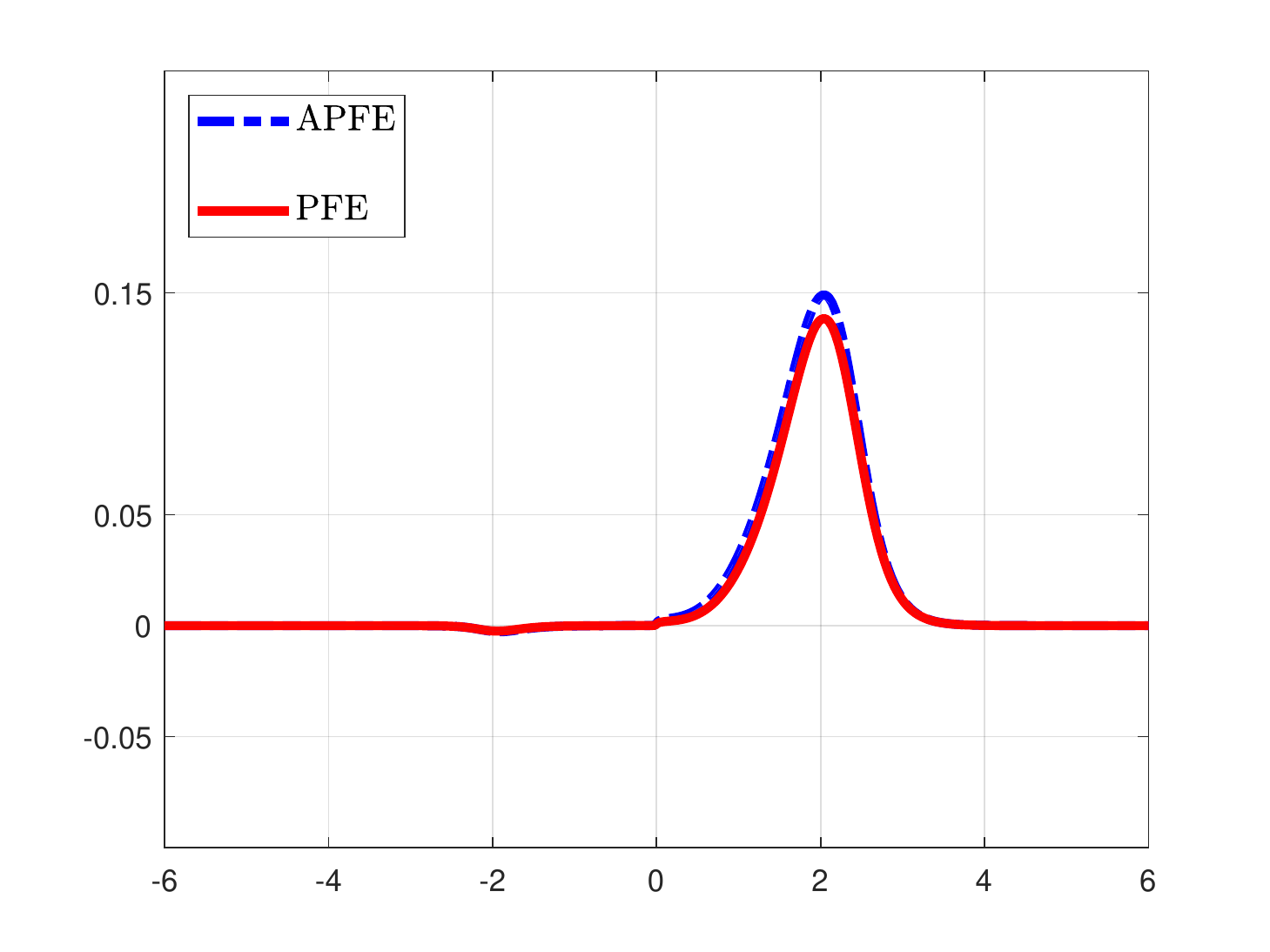}}
    \end{subfigures}
    \caption{Two-beam test comparison of standard projective scheme PFE and adaptive projective scheme APFE. Spatial discretization uses first order FORCE scheme. Spatially varying relaxation time $\epsilon_L = 10^{-4}$, $\epsilon_R = 10^{-2}$.}
    \label{fig:2beam}
\end{figure}

\subsection{Computational speedup of adaptive schemes}
\label{sec:results_speedup}
In this section, we give some results for the potential speedup of our new methods for a standard test case. As the focus of this paper is the derivation and analysis of the new schemes, we do not perform exhaustive numerical tests and simulations of all possible combinations of schemes, spatial discretizations and parameter settings, but leave this for future work.

The adaptive schemes in this paper are used to speed up the simulation of models with spectral gaps due to model differences throughout the computational domain, such that one (or more) stiff and one non-stiff domains are present. For the estimation of the speedup in comparison to a standard Forward Euler method (FE), we neglect the extrapolation steps of the PI methods and the boundary treatment. The speedup $S$ of a certain scheme with respect to a standard FE method is given by the ratio of the number of time steps $n$ over a unit time interval and can be computed according to \cite{Melis2019} as
\begin{equation}\label{e:speedup}
  S = \frac{n_{FE}}{n_{scheme}}
\end{equation}

For the different schemes, the number of time steps over a unit time interval is computed using the stability analysis from \Sects \ref{sec:standard_schemes} and \ref{sec:adaptive_schemes}. As an example, we consider the Upwind scheme and the largest possible time steps, to use explicit formulas for the speedup. We furthermore assume that a fraction of $\theta \in [0,1]$ of the computational domain uses the stiff relaxation time $\epsilon_L$, while the remaining $1-\theta$ are governed by the (also potentially) stiff relaxation time $\epsilon_R$.
\begin{itemize}
  \item[FE] $n_{FE} = \frac{1}{\Delta t}$, with time step size $\Delta t = CFL \frac{\Delta x}{\lambda_{max}}$ and $CFL = \frac{1}{\frac{\Delta x}{2 \epsilon_L \lambda_{max}}+1}$.
  \item[PFE] $n_{PFE} = \frac{K+1}{\Delta t}$, with time step size $\Delta t = CFL \frac{\Delta x}{\lambda_{max}}$ and $CFL = \frac{1}{\frac{\Delta x}{2 \epsilon_R \lambda_{max}}+1}$.
  \item[AFE] $n_{AFE} = \frac{\theta}{\delta t}+\frac{1-\theta}{\Delta t}$, with small time step size $\delta t = \frac{1}{\frac{\lambda_{max}}{\Delta x}+\frac{1}{2\epsilon_L}}$ and $\Delta t = CFL \frac{\Delta x}{\lambda_{max}}$ with $CFL = \frac{1}{\frac{\Delta x}{2 \epsilon_R \lambda_{max}}+1}$.
  \item[APFE] $n_{APFE} = \frac{\theta (K+1)}{\Delta t}+\frac{1-\theta}{\Delta t}$, with time step size $\Delta t = CFL \frac{\Delta x}{\lambda_{max}}$ and $CFL = \frac{1}{\frac{\Delta x}{2 \epsilon_R \lambda_{max}}+1}$.
  \item[APPFE] $n_{APPFE} = \frac{\theta (K_L+1)}{\Delta t}+\frac{(1-\theta)(K_R+1)}{\Delta t}$, with time step size $\Delta t = CFL \frac{\Delta x}{\lambda_{max}}$ and $CFL = 1$.
\end{itemize}

Note that the main gain for the speedup results from a less severe constraint on the CFL number.

As a numerical example we consider the base settings $\lambda_{max}=6$, $\Delta x = \frac{1}{50}$ and the projective schemes PFE, APFE, APPFE will use $K=K_L=K_R=1$. For the remaining parameters, we consider the following three scenarios:
\begin{itemize}
  \item[(A)] medium large spectral gaps on equally large domains $\epsilon_L = 10^{-4},\epsilon_R = 10^{-3}$, $\theta = \frac{1}{2}$.
  \item[(B)] large spectral gaps on equally large domains $\epsilon_L = 10^{-6},\epsilon_R = 10^{-4}$, $\theta = \frac{1}{2}$.
  \item[(C)] large spectral gaps and the stiffness only in a small domain $\epsilon_L = 10^{-6},\epsilon_R = 10^{-4}$, $\theta = \frac{1}{10}$.
\end{itemize}

\begin{table}[H]
    \centering
    \caption{Speedup of time integration schemes in comparison to standard FE scheme.}
    \label{tab:speedup}
    \begin{tabular}{c||c|c|c}
      case  & (A) & (B) & (C)  \\ \hline \hline
      FE    & 1.0 & 1.0 & 1.0  \\ \hline
      PFE   & 3.3 & 47.2 & 47.2  \\ \hline
      AFE   & 1.7 & 1.9 & 9.1  \\ \hline
      APFE  & 4.4 & 62.9 & 85.8  \\ \hline
      APPFE & 8.8 & 833.8 & 833.8  \\
    \end{tabular}
\end{table}
The speedup depending on the scenario and the time integration scheme is given in \Table \ref{tab:speedup}. While the standard PFE scheme already achieves a considerable speedup for cases with a large spectral gap, only the APFE and APPFE methods can make us of the full potential by treating both domains differently. The AFE method gives a speedup in comparison to the FE method, but does not overcome the stiff time step constraint in the stiff part of the domain. It is clear from \Table \ref{tab:speedup} that only the projective schemes PFE, APFE, and APPFE can achieve a significant speedup and adaptivity again drastically improves the performance of the projective schemes.     
\section{Conclusion}
\label{sec:conclusion}
In this paper, we developed and analyzed the first spatially adaptive projective integration schemes for stiff hyperbolic balance laws with spectral gaps to speed up standard time integrations schemes.

After introduction of the model PDEs exemplified by two models from rarefied gases, a detailed spectral analysis revealed the spectral gap for different spatial discretization schemes. The analytical derivation was validated by a numerical example that showed the accuracy of the derived eigenvalue bounds. After that, standard time integration schemes like the Forward Euler scheme or the Projective Forward Euler scheme, were analyzed and a prohibitive condition for the $CFL$ number was derived in case of large spectral gaps. The newly derived spatially adaptive time integration schemes were able to successively overcome these constraints on the time step size by using one scheme in the stiff region in combination with another scheme in the other region. The explicit formulation of the projective integration schemes allowed for an accurate analysis of the stability properties such that parameter bounds could also derived and validated for the adaptive schemes. Additionally, we outlined an extension towards higher-order time integration schemes or telescopic schemes with connected stability domains.


The results showed that the adaptive projective integration schemes achieved a high accuracy and a significant speedup that grows with the variations in the relaxation time.

The analysis in this paper allows for a promising extension towards higher-order methods via adaptive Projective Runge-Kutta schemes \cite{Lafitte2017} or adaptive Telescopic Projective Integration schemes \cite{Melis2019} in the future. Additionally, more numerical tests for applications need to be performed, e.g., for moment models and free-surface flows \cite{Koellermeier2020c}.     
\appendix

\section{Non-conservative Spatial Discretization}
\label{app:NC}
The general system \eqref{e:vars_system1D} can be discretized in space-time using cell-averages $\Var_i^{n}$ at cell $i$ for $i=1, \ldots, N_x$ and time step $n$ using the finite volume method in non-conservative form as follows
\begin{equation}\label{e:scheme}
    \Var_i^{n+1}=\Var_i^n-\frac{\Delta t}{\Delta x}\big(D_{i-1/2}^{+} + D_{i+1/2}^{-}\big) - \frac{1}{\epsilon(x_i)} \vect{S} \left( \Var_{i}^n \right),
\end{equation}
with fluctuations
\begin{equation}
    D_{i+1/2}^{\pm}=\Vect{A}_\Phi^\pm (\Var_i,\Var_{i+1})
\end{equation}
given by a polynomial viscosity method (PVM)
\begin{equation}
    \Vect{A}_\Phi^\pm (\Var_L,\Var_R) =  \frac{1}{2}\left( \Vect{A}_\Phi \cdot (\Var_R-\Var_L) \pm \Vect{Q}_\Phi \cdot (\Var_R-\Var_L)\right),  \label{fvm}
\end{equation}
with generalized Roe linearization $\Vect{A}_\Phi = \Vect{A}_\Phi(\Var_L,\Var_R)$ given by
\begin{equation}\label{e:Roe-gen}
    \Vect{A}_\Phi(\Var_L,\Var_R)\cdot (\Var_R-\Var_L)=\int_0^1 \Vect{A} (\Phi(s;\Var_L,\Var_R))\frac{\partial\Phi}{\partial s}(s;\Var_L,\Var_R)\,ds.
\end{equation}
Note that $\Phi(s;\Var_L,\Var_R)$ denotes a path connecting the left and right states at the cell interface, such that $\Phi(0;\Var_L,\Var_R)= \Var_L$ and $\Phi(1;\Var_L,\Var_R)= \Var_R$. The choice of paths has been studied in the literature and especially for moment systems of the forms \eqref{e:QBME_A} and \eqref{e:grad_A_HSM}, a linear path $\Phi(s;\Var_L,\Var_R)= \Var_R = \Var_L + s \left(\Var_R - \Var_L \right)$ was found to be sufficiently accurate \cite{Koellermeier2017d,Koellermeier2017a}.

The PVM method is using a viscosity matrix $\Vect{Q}_\Phi = \Vect{Q}_\Phi(\Var_L,\Var_R)$ depending on the left and right states. It has the form $\Vect{Q}_\Phi(\Var_i,\Var_{i+1}) = P^{i+1/2}\left(\Vect{A}_\Phi\left(\Var_i,\Var_{i+1}\right)\right)$, where $P^{i+1/2}\left(\Vect{A}_\Phi(\Var_L,\Var_R)\right)$ is a function of the generalized Roe matrix.

Many standard schemes can be written in the PVM form:
\begin{itemize}
  \item the Upwind or Roe scheme uses
        \begin{equation}\label{e:upwind_PVM}
            \Vect{Q}_\Phi(\Var_L,\Var_R) = \left|\Vect{A}_\Phi(\Var_L,\Var_R)\right|,
        \end{equation}
        which is not a polynomial in $\Vect{A}_\Phi(\Var_L,\Var_R)$ and can only be constructed given the full eigenstructure of the model.
  \item the Lax-Friedrichs scheme uses 
        \begin{equation}\label{e:LF_PVM}
            \Vect{Q}_\Phi(\Var_L,\Var_R) = \frac{\Delta x}{\Delta t} \Vect{I}.
        \end{equation}
        with $\lambda_N$ and $\lambda_1$ the largest and smallest eigenvalues, respectively, of the linearized Roe matrix at the cell interface.
  \item the FORCE scheme uses 
        \begin{equation}\label{e:FORCE_PVM}
            \Vect{Q}_\Phi(\Var_L,\Var_R) = \frac{\Delta x}{2 \Delta t} \Vect{I} + \frac{ \Delta t}{2 \Delta x} \Vect{A}_\Phi(\Var_L,\Var_R).
        \end{equation}
\end{itemize}

In general, the function should be as close as possible to the absolute value function while $P^{i+1/2}\left(x\right)\geq \left| x \right|$ is required for stability.\\

When considering a system with constant system matrix or a linearization of the system such that $\Vect{A}(\Var) = const$, e.g. \eqref{e:grad_A_HSM}, the non-conservative scheme shown here simplifies to $\Vect{A}_\Phi(\Var_L,\Var_R)= \Vect{A}$ and $\Vect{Q}_\Phi(\Var_i,\Var_{i+1})= \Vect{Q}$. Additionally, we assume a linear or linearized source term, e.g. \eqref{e:BGK_term}, such that $\vect{S} \left( \Var_{i}^n \right) = \Vect{S} \Var_{i}^n$, for constant matrix $\Vect{S} \in \mathbb{R}^{N \times N}$.

Note that a higher-order extension of the non-conservative scheme is possible as described in \cite{Koellermeier2020}.




\bibliographystyle{abbrv}
\bibliography{library_fixed}

\end{document}